\documentclass[10pt,twoside, a4paper,reqno]{amsart}

\usepackage{amsmath,amssymb,eucal,graphicx,mathtools,amsthm}
\usepackage{thmtools}
\usepackage[foot]{amsaddr}
\usepackage{hyperref}
\usepackage{graphicx}
\usepackage{caption}
\usepackage{subcaption}
\usepackage{listings}
\usepackage[figuresleft]{rotating}
\usepackage{lscape}
\usepackage{tikz}
\usetikzlibrary{shapes.geometric, positioning,calc, patterns,decorations.markings,backgrounds}
\usepackage{contour}
\usetikzlibrary{arrows}
\usetikzlibrary{decorations.pathmorphing}
\usepackage{tikz-cd}
\usepackage{stmaryrd}
\usepackage[capitalise]{cleveref}
\usepackage{comment}
\usepackage{verbatim}
\usepackage{tkz-euclide}
\usepackage{adjustbox}
\usepackage{mathdots}
\usepackage{circuitikz}
\usetikzlibrary{decorations.markings}
\usepackage{cite}
\usepackage{todonotes}
\usepackage{mathrsfs}
\usepackage{mathtools}
\usepackage{enumitem}
\usepackage[margin=2.5cm]{geometry}
\usepackage{varwidth}

\usepackage[all]{xy}
\xyoption{all}

\usepackage{microtype}

\mathchardef\mhyphen="2D 
\DeclareMathOperator{\mods}{mod}

\DeclareMathOperator{\Ext}{Ext}
\DeclareMathOperator{\Hom}{Hom}
\DeclareMathOperator{\RHom}{\textbf{R}Hom}

\DeclareMathOperator{\End}{End}

\DeclareMathOperator{\projdim}{pdim}
\DeclareMathOperator{\add}{add}

\DeclareMathOperator{\Gen}{Gen}

\DeclareMathOperator{\coker}{coker}

\DeclareMathOperator{\topp}{top}

\newcommand{\A}{\mathscr{A}}

\newcommand{\Ccal}{\mathcal{C}}
\newcommand{\Dcal}{\mathcal{D}}
\newcommand{\Ecal}{\mathcal{E}}
\newcommand{\Lcal}{\mathcal{L}}
\newcommand{\Rcal}{\mathcal{R}}
\newcommand{\Tcal}{\mathcal{T}}

\newcommand{\Trm}{\mathrm{T}}
\newcommand{\Prm}{\mathrm{P}}

\newcommand{\nsplit}{\textnormal{ns}}
\newcommand{\ssplit}{\textnormal{s}}

\newcommand{\Erm}{\mathrm{E}}

\counterwithin*{equation}{section}

\newtheorem{theorem}{Theorem}[section]

\newtheorem{lemma}[theorem]{Lemma}
\newtheorem{proposition}[theorem]{Proposition}

\theoremstyle{definition}
\newtheorem{definition}[theorem]{Definition}
\newtheorem{example}[theorem]{Example}
\newtheorem{remark}[theorem]{Remark}

\newcommand{\support}[1]{\color{red}}

\title[Exceptional versus $\tau$-exceptional sequences]{Exceptional versus $\tau$-exceptional sequences for the Auslander algebra of $K[x]/(x^t)$}

\author[M. Kaipel]{Maximilian Kaipel}
\address{Maximilian Kaipel, Fakultät für Mathematik, Universität Bielefeld, 33501 Bielefeld, Germany}
\email{mkaipel@math.uni-bielefeld}


\begin{document}

\begin{abstract}
	For $\A_t$, the Auslander algebra of $K[x]/(x^t)$, it is shown that every complete exceptional sequence of $\A_t$-modules is a complete $\tau$-exceptional sequence. Moreover, it is established that the mutation of complete $\tau$-exceptional sequences generalises the mutation of complete exceptional sequences in the category of $\A_t$-modules.
\end{abstract}

\maketitle

\section{Introduction}

Exceptional sequences are classical objects in algebraic geometry and representation theory, which were introduced by Rudakov's school in the study of vector bundles \cite{GorodentsevRudakov,Gorodentsev1988, Rudakov90}. These sequences live in the triangulated categories arising naturally, as bounded derived categories, in these mathematical areas. They consist of exceptional objects, that is, objects without self-extensions and with simple endomorphism algebras. An exceptional sequence is an ordering of such objects so that $\Hom$ and $\Ext$ vanish in one direction. The remarkable power of an exceptional sequence is most apparent when it generates the entire ambient triangulated category, in which case the sequence is called \textit{complete}. Complete exceptional sequences induce semi-orthogonal decompositions of triangulated categories \cite{Bondal1990} and can be used to construct stability conditions on triangulated categories \cite{Macri2007}.

For these reasons, the pursuit of triangulated categories admitting complete exceptional sequences has been a long-standing endeavour for researchers since the 1970s \cite{Beilinson1979,Kapranov1988, cbw92, Kawamata2006, PrabhuNaik2017, CHS2025}. Complete exceptional sequences are also closely connected to the theory of highest weight categories and quasi-hereditary algebras \cite{Krause2017}. One particularly important example of such a quasi-hereditary algebra is the Auslander algebra $\A_t$ of $K[x]/(x^t)$ for $t \geq 1$ and $K$ an algebraically closed field. This finite dimensional algebra arises when studying actions of linear groups on flags \cite[Sec. 4]{HilleRoehrle1999} and is related to chains of $(-2)$-curves on smooth projective surfaces \cite[Sec. 3.2]{HillePloog2019b}. Tilting modules and the quasi-hereditary structure of the algebra $\A_t$ were investigated in \cite{BHRR1999} and complete exceptional sequences in $\mods \A_t$, the category of finite dimensional right $\A_t$-modules, were characterised in \cite{HillePloog2019}. The algebra $\A_t$ is also closely related to the preprojective algebra of type $A$ \cite{RingelZhang2014}. 

Inspired by the introduction of cluster algebras \cite{FZ2002} and their associated combinatorics, the theory of tilting modules over finite dimensional algebras was extended with great success using Auslander--Reiten theory \cite{AIR2014}. Within this new framework of $\tau$-tilting theory, many classical concepts from representation theory have successfully been generalised to all finite dimensional algebras. Especially the combinatorial aspects of classical tilting theory behave best for hereditary algebras, but generally break beyond this setting. Here, support $\tau$-tilting modules generalise tilting modules and resolve this issue. 

Complete exceptional sequences of modules exist for hereditary algebras \cite{cbw92} but generally fail to exist for arbitrary finite dimensional algebras. To always have complete sequences, the authors of \cite{BuanMarsh2018t} introduced $\tau$-exceptional sequences and showed that these two notions coincide for hereditary algebras. However, in contrast to the relationship between tilting and $\tau$-tilting modules, a complete exceptional sequence (if it exists) is not necessarily a complete $\tau$-exceptional sequence, see \cref{rmk:excepnottau}. In fact, essentially no relationship between the two types of sequences is known beyond hereditary algebras. The only exception are tensor algebras of hereditary path algebras with finite dimensional commutative local $K$-algebras, whose $\tau$-exceptional sequences are in bijection with those of the hereditary path algebra \cite{Nonis2025}. The first main theorem of this paper establishes such a relationship for the algebra $\A_t$.

\begin{theorem}(\cref{thm:excepistauexcep})\label{thm:introthm1}
    Every complete exceptional sequence in $\mods \A_t$ is a $\tau$-exceptional sequence. More precisely, there is an explicit bijection 
    \[ \Phi: \{ \text{basic tilting $\A_t$-modules} \} \longrightarrow \{ \text{complete exceptional sequences in $\mods \A_t$} \}, \]
    given by the bijection between TF-ordered $\tau$-tilting modules and complete $\tau$-exceptional sequences of \cite{MendozaTreffinger}.
\end{theorem}

With this result as a starting point it is natural to compare the mutation theories of complete exceptional sequences and complete $\tau$-exceptional sequences. Both mutation operations are defined via a mutation operation on two consequence terms in a complete ($\tau$-)exceptional sequence, so called ($\tau$-)exceptional pairs. The mutation of exceptional sequences is classical but various questions, especially regarding its transitivity remained open until very recently \cite{CHS2025}. In general, the mutation of an exceptional sequence in $\mods \A_t$ does not remain a sequence of objects in $\mods \A_t$, but rather becomes a sequence of objects in the bounded derived category $\Dcal^b(\mods \A_t)$. In contrast, the mutation of $\tau$-exceptional sequences always remains in $\mods \A_t$. Therefore, the comparison in this article restricts itself to the case where the mutation of an exceptional sequence in $\mods \A_t$ remains in $\mods \A_t$. The mutation of $\tau$-exceptional sequences was recently introduced in \cite{BHM2024} and has already inspired substantial work. Its transitivity has been established for a large class of finite dimensional algebras \cite{BHM2025} and investigations of whether the mutation of $\tau$-exceptional sequences satisfies braid group relations have yielded insightful results \cite{KaipelTerland2025, Nonis2025b}. The second main theorem establishes that the mutation of complete exceptional sequences which remains in $\mods \A_t$ is a particular case of the mutation of complete $\tau$-exceptional sequences.

\begin{theorem}(\cref{thm:mutthm})\label{thm:introthm2}
    Let $\Ecal= (E_t, \dots, E_1)$ be a complete exceptional sequence in $\mods \A_t$ and let $i \in \{2, \dots, t\}$. If the left mutation of $\Ecal$ as an exceptional sequence (at position $i$) lies in $\mods \A_t$ then it coincides with the left mutation of $\Ecal$ as a $\tau$-exceptional sequence (at position $i$).
\end{theorem}

Tilting modules, which are in bijection with complete exceptional sequences by \cref{thm:introthm1}, have their own mutation theory, which was generalised by that of support $\tau$-tilting modules. It turns out that the mutation of tilting modules, which replaces one indecomposable summand of a tilting module by another, is the lift of the mutation of complete exceptional sequences under the inverse of $\Phi$. In other words, $\Phi$ commutes with the mutations. A similar phenomen was also observed in \cite[Sec. 9]{Geuenich2022}.

\begin{proposition}(\cref{prop:mutcomp}, simplified)\label{prop:introprop3}
	The explicit bijection $\Phi$ of \cref{thm:introthm1} is compatible with the mutation of tilting $\A_t$-modules and the mutation of complete exceptional sequences in $\mods \A_t$.
\end{proposition}

To summarise the results above, let $t \geq 2$ and let $i \in \{2, \dots, t\}$. \cref{thm:introthm1} establishes the top row of the diagram in \cref{fig:introfig}. Take a basic tilting $\A_t$-module $T$. If there is another tilting module $T'$ which differs from $T$ only in one direct summand, or equivalently if the mutation of its corresponding complete exceptional sequence lies in $\mods \A_t$, then  \cref{thm:introthm2} and \cref{prop:introprop3} imply that the entire commutative diagram of \cref{fig:introfig} exists. In \cref{fig:introfig}, $\mu$ denotes the mutation of tilting modules, $\psi$ denotes the mutation of exceptional sequences and $\varphi$ denotes the mutation of $\tau$-exceptional sequences. These statements hold both for left and for right mutations, because these are inverse procedures.
\begin{figure}[ht!]
\[
\begin{tikzcd}
\{\text{basic tilting $\A_t$-modules}\} \arrow[d, "\mu_{i-1}"] \arrow[r, "\Phi","\text{bij.}"'] & \left\{ \begin{varwidth}{10em} \begin{center} complete exceptional sequences in $\mods \A_t$ \end{center} \end{varwidth} \right\} \arrow[d, "\psi_i"] \arrow[r, phantom, "\subsetneq"]& \left\{ \begin{varwidth}{10em} \begin{center} complete $\tau$-exceptional sequences in $\mods \A_t$ \end{center} \end{varwidth} \right\}  \arrow[d, "\varphi_i"]  \\
\{\text{basic tilting $\A_t$-modules}\} \arrow[r, "\Phi","\text{bij.}"'] & \left\{ \begin{varwidth}{10em} \begin{center} complete exceptional sequences in $\mods \A_t$ \end{center} \end{varwidth} \right\} \arrow[r, phantom, "\subsetneq"] &  \left\{ \begin{varwidth}{10em} \begin{center} complete $\tau$-exceptional sequences in $\mods \A_t$ \end{center} \end{varwidth} \right\}
\end{tikzcd}
\]
\caption{Overview of the relationship between tilting modules, exceptional sequences and $\tau$-exceptional sequences in $\mods \A_t$ and their mutations.}
\label{fig:introfig}
\end{figure}

To illustrate this behaviour explicitly, consider the following example.

\begin{example}\label{exmp:t2}
    Consider the case $t=2$. The case $t=3$ is illustrated in \cref{sec:example}. The algebra $\A_2$ is isomorphic to the bound path algebra
    \[ K \left( \begin{tikzcd}
        1 \arrow[r, "a", shift left, bend left] & 2 \arrow[l, "b", shift left, bend left]
    \end{tikzcd} \right) / \langle ab \rangle .\]
    The algebra $\A_2$ has 5 indecomposable modules up to isomorphism, all of which are uniserial. This makes it a Nakayama algebra, which allows for a simple calculation of its support ($\tau$-)tilting modules and ($\tau$-)exceptional sequences. The left mutation $\varphi$ of $\tau$-exceptional sequences for Nakayama algebras has been described in simple terms in \cite[Thm. 1.3]{BKT2025}. In \cref{fig:introfigex}, the tilting $\A_2$-modules are the subset of all support $\tau$-tilting modules highlighted in {\color{purple} \textbf{purple}} and the complete exceptional sequences in $\mods \A_2$ are the subset of all complete $\tau$-exceptional sequences also highlighted in {\color{purple} \textbf{purple}}. 
    \begin{figure}[ht!]
    \[
    \begin{tikzcd}[ampersand replacement=\&, column sep=5]
        \& {\color{purple}\boldsymbol{\begin{smallmatrix}&2\\1&\\&2\end{smallmatrix} \oplus \begin{smallmatrix}1&\\&2\end{smallmatrix}}} \arrow[ld, thick, purple, "\mu_1",swap] \arrow[rd] \\
        {\color{purple}\boldsymbol{\begin{smallmatrix}&2\\1&\\&2\end{smallmatrix} \oplus \begin{smallmatrix}2\end{smallmatrix}}} \arrow[d] \& \& \begin{smallmatrix}1\end{smallmatrix} \oplus \begin{smallmatrix}1&\\&2\end{smallmatrix} \arrow[d]\\
        \begin{smallmatrix}2\end{smallmatrix} \arrow[rd] \& \& \begin{smallmatrix}1\end{smallmatrix} \arrow[ld] \\
        \& \begin{smallmatrix}0 \end{smallmatrix} 
    \end{tikzcd}
    \qquad  \qquad
    \begin{tikzcd}[ampersand replacement=\&, column sep=5]
        \& {\color{purple}\boldsymbol{\left(\begin{smallmatrix}2\end{smallmatrix}, \begin{smallmatrix}1&\\&2\end{smallmatrix} \right)}} \arrow[ld, "\varphi_2"', "\psi_2", thick, purple] \\
        {\color{purple}\boldsymbol{\left(\begin{smallmatrix}&2\\1&\end{smallmatrix}, \begin{smallmatrix}2\end{smallmatrix} \right)}} \arrow[rd, "\varphi_2",swap] \& \& \left(\begin{smallmatrix}1&\\&2\end{smallmatrix}, \begin{smallmatrix}1\end{smallmatrix} \right) \arrow[lu, "\varphi_2",swap]\\
        \& \left(\begin{smallmatrix}1\end{smallmatrix}, \begin{smallmatrix}&2\\1&\\&2\end{smallmatrix} \right) \arrow[ru, "\varphi_2",swap]
    \end{tikzcd}
    \]
    \caption{Hasse diagram of support $\tau$-tilting modules (left) with tilting modules highlighted in {\color{purple}\textbf{purple}} and left mutation graph of complete $\tau$-exceptional sequences (right) with exceptional sequences highlighted in {\color{purple}\textbf{purple}}.}
    \label{fig:introfigex}
    \end{figure}
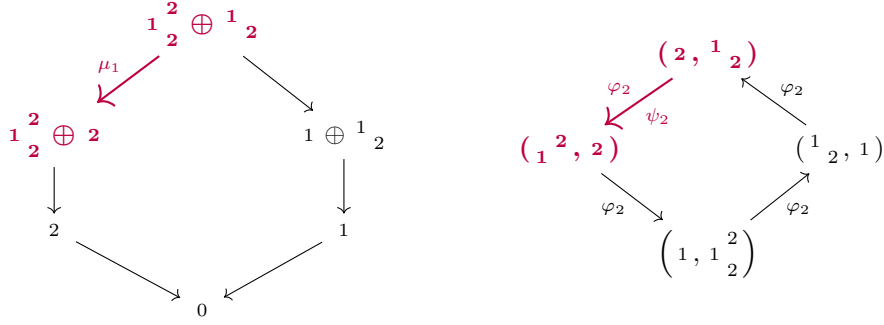
\end{example}

The structure of this paper is as follows: \cref{sec:prelim} recalls the preliminaries of ($\tau$-)tilting theory and collects some special properties of tilting $\A_t$-modules. Exceptional and $\tau$-exceptional sequences are also defined there. \cref{sec:mainthm1} is relatively short and culminates in the proof \cref{thm:introthm1}. \cref{sec:mutation} then introduces and compares the three different kinds of mutations, ultimately proving \cref{thm:introthm2} and \cref{prop:introprop3}. The final \cref{sec:example} illustrates the developed results for $\A_3$. 

\subsection*{Acknowledgements}
The author thanks Håvard U. Terland for sharing a QPA programme with him, which significantly helped the author's calculations of complete $\tau$-exceptional sequences and their mutations. The author is grateful to Iacopo Nonis for carefully reading an earlier version of this manuscript, which helped to improve the quality of this paper. The author acknowledges support from the Deutsche Forschungsgemeinschaft (DFG, German Research Foundation) as part of SFB-TRR 191 (Project ID 281071066) and as part of SFB-TRR 358/1 2023 (Project ID 491392403).

\section{Preliminaries}\label{sec:prelim}
Throughout, let $K$ be an algebraically closed field and, for a natural number $t \geq 1$, let $\A_t$ denote the Auslander algebra of $K[x]/(x^t)$. Note that all results are trivial for $t=1$. In other words, let $\A_t = \End(\bigoplus_{i=1}^t M(i))$, where $M(i)$ is the isomorphism class of the indecomposable $K[x]/(x^t)$-module of length $i$. The algebra $\A_t$ is isomorphic to the bound path $K$-algebra of the quiver
\[ \begin{tikzcd}
    1 \arrow[r, "a_1", shift left] & 2 \arrow[r, "a_2", shift left] \arrow[l, "b_1", shift left] & \dots \arrow[r, "a_{t-1}", shift left] \arrow[l, "b_2", shift left] & t \arrow[l,"b_{t-1}", shift left]
\end{tikzcd}, \quad \text{with relations}\quad  a_1 b_1= 0, a_{i+1}b_{i+1}= b_i a_i, \text{ for } 1 \leq i \leq t-2.\]

As an Auslander algebra of a representation finite algebra, the algebra $\A_t$ is of global dimension at most 2 \cite[Thm. 4.6]{Auslander1974b}. The category of finite dimensional right $\A_t$-modules is denoted by $\mods \A_t$ and given $M \in \mods \A_t$, the number of isomorphism classes of indecomposable direct summands of $M$ is denoted by $|M|$. There is a sequence of inclusions of indecomposable projective $\A_t$-modules
\[ P(1) \subset P(2) \subset \dots \subset P(t) \] 
which induces a (unique) quasi-hereditary structure on $\A_t$. In this paper, various special families of $\A_t$-modules play central roles. This section introduces these and collects their most relevant properties.

\begin{definition}\label{defn:tilting}
    Let $T \in \mods \A_t$. Then $T$ is called \textit{partial tilting} if $\Ext^1(T,T) = 0$ for $i \geq 1$ and $\projdim T \leq 1$. It is called \textit{tilting} if additionally $|T|=t$.
\end{definition}

For the algebra $\A_t$, it turns out that studying its maximal ideals provides an insightful perspective on its tilting modules. Consider the maximal ideal $I_i \coloneqq \langle 1-e_i \rangle \subseteq \A_t$, where $\{e_i\}_{i=1}^t$ is a complete set of primitive orthogonal idempotents of $\A_t$ and let $\langle I_1, \dots, I_{t-1} \rangle$ denote the semigroup generated multiplicatively by the maximal ideals $\{I_i\}_{i=1}^{t-1} \subseteq \A_t$. 

\begin{theorem}\cite[Thm. 1.1(1)]{IyamaZhang2020}\label{thm:idealsemigroup}
    The set of basic tilting $\A_t$-modules is given by $\langle I_1, \dots, I_{t-1} \rangle$. 
\end{theorem}

For the algebra $\A_t$, many properties of its tilting modules can be deduced from properties of its projective modules as a consequence of the following theorem.

\begin{theorem}\cite[Thm. 4]{BHRR1999}\label{thm:tiltingendo}
    For any basic tilting $\A_t$-module $T$, there is an isomorphism $\End(T) \simeq \A_t$. 
\end{theorem}

This result can also be found in \cite[Thm. 3.6]{IyamaZhang2020}. The following enumeration of basic tilting $\A_t$-modules is a key ingredient in the proof of \cref{thm:introthm1}.

\begin{theorem}\cite[Thm. 3]{BHRR1999}\label{thm:BHRRcount}
    There is a bijection between the symmetric group $\mathfrak{S}_t$ and the set of basic tilting $\A_t$-modules. Consequently, there are $t!$ basic tilting $\A_t$-modules.
\end{theorem}

Importantly, there are also $t!$ complete exceptional sequence in $\mods \A_t$ \cite{HillePloog2019} and these bijections are also discussed in \cite[Sec. 9]{Geuenich2022}. Generalised tilting complexes in the bounded homotopy category of $\A_t$ are characterised in \cite{Geuenich2022, Sauter2019}.

\begin{definition}
	Let $E \in \mods \A_t$, then $E$ is called \textit{exceptional} if $\Ext^i(E,E) = 0$ for $i \geq 1$ and $\End(E) \simeq K$. A sequence $(E_r, \dots, E_1)$ of exceptional $\A_t$-modules is called an \textit{exceptional sequence} if $\Hom(E_i, E_j) = 0 = \Ext^{\ell}(E_i, E_j)$ for $1 \leq i < j \leq r$. It is called a complete exceptional sequence if $r=t$.
\end{definition}

The goal of this paper is to compare these classical notions with modern adaptations in the context of $\tau$-tilting theory \cite{AIR2014}. A red thread commonly seen throughout $\tau$-tilting theory uses the Auslander--Reiten translation $\tau$ in $\mods \A_t$ to replace $\Ext^1(!,?)$ vanishing conditions by $\Hom(?, \tau(!))$ vanishing conditions.

\begin{definition}\cite{AIR2014}
    Let $T \in \mods \A_t$, then $T$ is called \textit{$\tau$-rigid} if $\Hom(T, \tau T) = 0$ and it is called \textit{$\tau$-tilting} if additionally $|T|=t$. A $\tau$-rigid module $T$ is called \textit{support $\tau$-tilting} if it is $\tau$-tilting in $\A_t/\langle e \rangle$ for some idempotent $e \in \A_t$. A pair $(T,P)$ of $\A_t$-modules is called \textit{$\tau$-rigid} if $T$ is $\tau$-rigid and $\Hom(P,T) =0$. Sometimes a $\tau$-rigid pair $(T,P)$ is denoted by $T \oplus P[1]$. 
\end{definition}

It is a well-known fact that partial tilting objects are $\tau$-rigid, which makes it immediate that tilting objects are $\tau$-tilting.

\begin{lemma}\label{lem:tiltistau}
    Every partial tilting $\A_t$-module is $\tau$-rigid. Conversely, every $\tau$-rigid module of projective dimension at most 1 is partial tilting.
\end{lemma}
\begin{proof}
    If $\projdim T \leq 1$ then $\Ext^1(T,T)=0$ is equivalent to $\Hom(T, \tau T)=0$ by \cite[Cor. IV.4.7]{ARS1995}.
\end{proof}

By \cite[Thm. 0.5]{AIR2014}, associated to every $\tau$-rigid module $T$ in $\mods \A_t$ are a torsion class $\Gen T$ and its corresponding torsion-free class $T^\perp$, which are full subcategories defined by
\[ \Gen T \coloneqq \{ M \in \mods \A_t: \exists T^r \twoheadrightarrow M \to 0 \} \]
\[ T^\perp \coloneqq \{ M \in \mods \A_t: \Hom(T,M) = 0\}. \]  

Together, they form a torsion pair $(\Gen T, T^\perp)$ in $\mods \A_t$ which comes with two associated maps of objects $t_T: \mods A \to \mods A$ and $f_T: \mods \A_t \to \mods \A_t$, see \cite{Dickson66}. These maps are such that for every $M \in \mods A$ there is a unique (up to isomorphism) short exact sequence
\[ 0 \to t_T (M) \to M \to f_T (M) \to 0\]
satisfying $t_T (M) \in \Gen T$ and $f_T (M) \in T^\perp$. The full subcategory $^\perp T$ is defined dually to $T^\perp$. This leads to the \textit{$\tau$-perpendicular} subcategory $J(T) \coloneqq T^\perp \cap {}^\perp \tau T$. If $(T,P)$ is $\tau$-rigid, set $J(T,P) \coloneqq T^\perp \cap {}^\perp \tau T \cap P^\perp$. Then $J(T,P)$ is an extension-closed full subcategory of $\mods \A_t$ and abelian, whence it has an Auslander--Reiten translation and $\tau$-rigid modules which are possibly distinct from those in $\mods \A_T$, see \cite{Jasso2015}. 

\begin{remark}\label{rmk:notationjustification}
    Let $T$ be $\tau$-rigid. The map $f_T(-)$ defines a bijection 
    \begin{align*} \{ \text{indecomposable } M \in \mods \A_t : M \oplus T \text{ is $\tau$-rigid in $\mods \A_t$ and } M \not \in \Gen T \} \\
    \to \{\text{indecomposable } N \in J(T) : N \text{ is $\tau$-rigid in $J(T)$}\}
     \end{align*}
     by \cite[Prop. 4.5]{BuanMarsh2018t}. By a slight abuse of notation, the inverse assignment of this bijection is denoted $f_T^{-1}(-)$. This is used heavily in \cref{sec:mutation}, where this notation is not used in any other context than to refer to this bijection.
\end{remark}

The observation that $\tau$-rigid modules in $J(T)$ may not necessarily be $\tau$-rigid in $\mods \A_t$ is fundamental in the following recursive definition. 
\begin{definition}\cite[Def. 1.3]{BuanMarsh2018t}
	A sequence $(M_r, \dots, M_1)$ of indecomposable $\A_t$-modules is called \textit{$\tau$-exceptional} if $M_1$ is $\tau$-rigid in $\mods \A_t$ and $(M_r, \dots, M_{2})$ is a $\tau$-exceptional sequence in $J(M_1)$. The sequence is called \textit{complete} if $r=t$.
\end{definition}

In particular, the modules $M_r, \dots, M_2$ need not be $\tau$-rigid in $\mods \A_t$. To understand the relationship between these sequences and $\tau$-rigid modules requires the following definition.

\begin{definition}
    Let $T$ be a $\tau$-rigid module and $T_r \oplus \dots \oplus T_1$ be an ordered decomposition of $T$ into indecomposable direct summands. Then $T_r \oplus \dots \oplus T_1$ is called a \textit{TF-ordered decomposition} or a \textit{TF-ordered} $\tau$-rigid module if $T_i \not \in \Gen (\bigoplus_{j=1}^{i-1} T_j)$ for $2 \leq i \leq r$.
\end{definition}

Using the bijection of \cref{rmk:notationjustification} enables the passage between these ordered decompositions and $\tau$-exceptional sequences.

\begin{theorem}\cite[Thm. 5.1]{MendozaTreffinger}\label{thm:MendozaTreff}
    Let $T = T_r \oplus \dots \oplus T_1$ be a TF-ordered $\tau$-rigid module. Then 
    \[ \Phi(T) = \left(f_{\bigoplus_{i=1}^{r-1}(T_i)} T_r ,\dots,f_{T_1} (T_2), T_1 \right)\]
    is a $\tau$-exceptional sequence of length $r$. Moreover, every $\tau$-exceptional sequence of length $r$ arises this way. 
\end{theorem}

To fully justify why this is well-defined, it is necessary to state that $J(T_2 \oplus T_1) = J^{J(T_1)}(f_{T_1} T_2)$ and that this generalises to decompositions with more terms by \cite[Thm. 1.4]{BuanMarsh2021w}, see also \cite[Thm. 6.4]{BuanHanson2023}. In this way, define the subcategory
\[ J(f_{T_1} (T_2), T_1) \coloneqq J(T_2 \oplus T_1), \]
and for a $\tau$-exceptional sequence $(M_r, \dots, M_1)$ define $J(M_r, \dots, M_1)$ analogously via \cref{thm:MendozaTreff}.

\section{Comparing exceptional and $\tau$-exceptional sequences}\label{sec:mainthm1}

Let $e_1, \dots, e_t$ be the primitive orthogonal idempotents of $\A_t$, which correspond to the modules $M(i)$ over $K[x]/(x^t)$ and equivalently to the vertices of the quiver presentation described in \cref{sec:prelim}. Using these, any basic tilting $\A_t$-module $T$ admits an ordered decomposition $T = e_tT \oplus \dots \oplus e_1 T$ into indecomposable modules, which will be referred to as its \textit{quasi-hereditary decomposition}. Throughout, this ordered decomposition will be written as $T= T_t \oplus \dots \oplus T_1$. The following is dual to \cite[Lem. 4.3]{IyamaZhang2020}.

\begin{lemma}\label{lem:tiltapprox}
    Let $T$ be a basic tilting $\A_t$-module. 
    \begin{enumerate} 
    \item The minimal right $\add( \bigoplus_{j=1}^{i-1} T_j)$-approximation of $T_i$ is given by $h_{i-1}: T_{i-1} \to T_i$, which is the left multiplication of the arrow $b_{i}: i \to i-1$ in the quiver of $\A_t$. This is a monomorphism. 
    \item The minimal right $\add( \bigoplus_{j=i+1}^{t} T_j)$-approximation of $T_i$ is given by $h'_{i-1}: T_{i} \to T_{i-1}$, which is the left multiplication of the arrow $a_{i+1}: i+1 \to i$ in the quiver of $\A_t$. 
    \end{enumerate}
\end{lemma}
\begin{proof}
    Left multiplication defines an isomorphism $\A_t \to \End(T)$, by \cref{thm:tiltingendo}. This implies that $\Hom(T,-): \add (T) \to \add (\A_t)$ is an equivalence of categories. The minimal right $\add(\bigoplus_{j=1}^{i-1} e_j \A_t)$-approximation of $e_i \A_t$ is $e_{i-1}\A_t \to e_i \A_t$, which is given by left multiplication by $b_i$. The first assertion follows and the second assertion is proved analogously. Since left multiplication by $b_i$ gives a monomorphism $P(i-1) \to P(i)$, its restriction $h_{i-1}$ is also a monomorphism. 
\end{proof}

As a consequence of the previous lemma, the following simplification of the torsion-free map may be obtained.
\begin{lemma}\label{lem:torfdesc}
	Let $T$ be a basic tilting $\A_t$-module and let $T= T_t \oplus \dots \oplus T_1$ be its TF-ordered decomposition. There is an isomorphism
	\[ f_{\bigoplus_{j<i} T_j} (T_i) \simeq T_i/T_{i-1} \]
	for all $i \in \{2, \dots, t\}$.
\end{lemma}
\begin{proof}
	By \cref{lem:tiltapprox}(1) and the defining property of the torsion pair $(\Gen T, T^\perp)$, the solid part of the following diagram exists
	\[
	\begin{tikzcd}
	& ( \bigoplus_{j <i } T_j)^r \arrow[d, twoheadrightarrow, "\epsilon", swap] \arrow[r, "\gamma", dashed] & T_{i-1} \arrow[d, "h_{i-1}", hookrightarrow] \arrow[ld, "\delta",swap, dashed] \\
	0 \arrow[r] & t_{\bigoplus_{j<i} T_j} (T_i) \arrow[r, "\beta", hookrightarrow] & T_i \arrow[r, twoheadrightarrow] &f_{\bigoplus_{j<i} T_j} (T_i) \arrow[r] & 0 
	\end{tikzcd},
	\] 
	for some $r \geq 1$, where $h_{i-1}$ is a minimal right $\add (\bigoplus_{j<i} T_j)$-approximation and $\beta$ is a minimal right $\Gen(\bigoplus_{j<i} T_j)$-approximation. Because $h_{i-1}$ is a right $\add (\bigoplus_{j<i} T_j)$-approximation and $\beta \circ \epsilon$ is a morphism from $(\bigoplus_{j<i} T_j)^r$ to $T_i$, there exists a map $\gamma: (\bigoplus_{j<i} T_j)^r \to T_{i-1}$ such that $h_{i-1} \circ \gamma = \beta \circ \epsilon$. Similarly, because $\beta$ is a right $\Gen(\bigoplus_{j<i} T_j)$-approximation and $T_{i-1} \in \Gen(\bigoplus_{j<i} T_j)$, it follows that there exists $\delta: T_{i-1} \to t_{\bigoplus_{j<i} T_j} (T_i)$ such that $h_{i-1} = \beta \circ \delta$. As $h_{i-1}$ is a monomorphism, $\delta$ has to be a monomorphism as well. Furthermore, it follows that
	\[ \beta \circ \delta \circ \gamma = h_{i-1} \circ \gamma = \beta \circ \epsilon \quad \Longrightarrow \quad \delta \circ \gamma = \epsilon \]
	because $\beta$ is a monomorphism. Since $\epsilon$ is an epimorphism, this means that $\delta$ has to be an epimorphism. In conclusion $\delta$ is an isomorphism, from which the result follows.
\end{proof}

The following provides the first step towards establishing a relationship between exceptional and $\tau$-exceptional sequences in $\mods \A_t$.

\begin{lemma}\label{lem:tfordered}
    Let $T$ be a basic tilting $\A_t$-module. Then its quasi-hereditary decomposition is a TF-ordered $\tau$-tilting module.
\end{lemma}
\begin{proof}
	Clearly, $T$ is a $\tau$-tilting module by \cref{lem:tiltistau}. Assume for a contradiction that $T_i \in \Gen(\bigoplus_{j=1}^{i-1} T_j)$ for some $i \in \{2, \dots, t\}$. Then there exists an epimorphism $g: (\bigoplus_{j=1}^{i-1} T_j)^r \to T_i$ for some $r \geq 1$. However, every component $g_\ell$ of $g$ factors through the monomorphism $h_{i-1}: T_{i-1} \to T_i$ by \cref{lem:tiltapprox}(1). Therefore, $g$ cannot be an epimorphism. 
\end{proof}

To establish a connection between exceptional modules and $\tau$-rigid modules, the following characterisation of exceptional modules is important. Recall that a module $M \in \mods \A_t$ is called \textit{thin} if its dimension vector only has entries equal to 0 or 1. 

\begin{theorem}\cite[Thm. 2.2(1)]{HillePloog2019}\label{thm:HPcharexceptional}
	Let $M \in \mods \A_t$. Then $M$ is exceptional if and only if $M$ is indecomposable, thin and satisfies $\dim \Hom(P(t),M)=1$.
\end{theorem}

The final condition in the preceding theorem is related to the following.

\begin{lemma}\label{lem:tiltdimhom}
    Let $T$ be a basic tilting $\A_t$-module. Then $\dim \Hom(P(t), T_i)=i$ for all $i \in \{1, \dots, t\}$. 
\end{lemma}
\begin{proof}
    Let $i \in \{1, \dots, t\}$. From the description of $T$ as an element of the ideal semigroup in \cref{thm:idealsemigroup}, it follows directly that each $T_i$ is a submodule of $e_i \A_t = P(i)$. The indecomposable projective $P(i)$ is such that $\dim \Hom(P(t), P(i))=i$ for all $1 \leq i \leq t$. Therefore $i$ is an upper bound for $\dim \Hom(P(t), T_i)$. 

    Observe that $e_t \in I_j$ for all $j \in \{1, \dots, t-1\}$. It follows that $T_t = P(t)$, this establishes the result for $i=t$. Otherwise, let $X = \langle e_t \rangle \subseteq \A_t$ be the ideal generated by the idempotent at vertex $t$. It follows that $X$ is a submodule of $T$ and that $e_i X$ is a submodule of $T_i$. It is readily checked that $\dim \Hom(P(t), e_iX) = i$, which implies that $i$ is also a lower bound for $\dim \Hom(P(t), T_i)$. The result follows.
\end{proof}

It is now possible to show the first main result of this paper, establishing that every complete exceptional sequence in $\mods \A_t$ is also a complete $\tau$-exceptional sequence.

\begin{theorem}\label{thm:excepistauexcep}
    The bijection $\Phi$ of \cref{thm:MendozaTreff} restricts to a bijection between the quasi-hereditary decompositions of basic tilting $\A_t$-modules and complete exceptional sequences in $\mods \A_t$:
    \[ \Phi: 
    \begin{tikzcd}
\{\text{TF-ordered $\tau$-tilting $\A_t$-modules}\} \arrow[r, "\text{bij.}"] & \left\{ \begin{varwidth}{10em} \begin{center} complete $\tau$-exceptional sequences in $\mods \A_t$ \end{center} \end{varwidth} \right\} \\
\left\{ \begin{varwidth}{15em} \begin{center} quasi-hereditary decompositions of tilting $\A_t$-modules \end{center} \end{varwidth} \right\} \arrow[r, "\Phi","\text{bij.}"'] \arrow[u, phantom, "\cup"] & \left\{ \begin{varwidth}{10em} \begin{center} complete exceptional sequences in $\mods \A_t$ \end{center} \end{varwidth} \right\} \arrow[u, phantom, "\cup"]
\end{tikzcd}
    \]
    In particular, every complete exceptional sequence in $\mods \A_t$ is a complete $\tau$-exceptional sequence. 
\end{theorem}
\begin{proof}
    By \cref{lem:tfordered}, the quasi-hereditary decomposition of $T$ is a TF-ordered $\tau$-tilting module and hence lies in the domain of $\Phi$, whose codomain is the set of $\tau$-exceptional sequences. Since $|T|=t$, the image of $T$ under $\Phi$ is thus a complete $\tau$-exceptional sequence. Once it is established that $\Phi(T)$ is a complete exceptional sequence, the fact that there are $t!$ basic tilting modules by \cref{thm:BHRRcount} and $t!$ complete exceptional sequences by \cite[Introduction Thm.]{HillePloog2019} shows that all complete exceptional sequences must arise in this way because $\Phi$ is a bijection. 
    
    It remains to show that $\Phi(T_t \oplus \dots \oplus T_1)$ is a complete exceptional sequence. By \cref{lem:torfdesc}, the isomorphism
    \[ f_{\bigoplus_{j=1}^{i-1} T_j} (T_i) \simeq T_i/T_{i-1} \] 
	holds for all $i \in \{2, \dots, t\}$. It follows that 
	\[ \Phi(T) = (T_t/T_{t-1}, \dots, T_2/T_1, T_1). \]
	
	Moreover, $T$ defines a filtration $0 \subset T_1 \subset T_2 \subset \dots \subset T_{t-1} \subset T_t = P(t)$. The result \cite[Prop. 3.5]{HillePloog2019} then states that $\Phi(T)$ is a complete exceptional sequence if 
    \begin{enumerate}
        \item $\dim \Hom(P(t), T_i)=i$;
        \item $T_1$ is an exceptional module;
        \item each $T_i /T_{i-1}$ for $i \in \{ 2, \dots, t\}$ is an exceptional module.
    \end{enumerate}
    It remains to prove (2) and (3), as (1) is shown in \cref{lem:tiltdimhom}. Observe that the indecomposable projective $P(1)$ is thin. Thus, its submodule $T_1$ is also thin, which yields that $T_1$ is an exceptional modules by \cref{thm:HPcharexceptional} and \cref{lem:tiltdimhom}. It remains to show (3).
    
    Let $i \in \{2, \dots, t\}$. By \cref{thm:HPcharexceptional} it is necessary and sufficient to show that each $T_i/T_{i-1}$ is indecomposable, thin and such that $\dim \Hom(P(t), T_i/T_{i-1})=1$. Consider the short exact sequence
    \begin{equation}\label{eq:sesexceptional} 0 \to T_{i-1} \xrightarrow{h_{i-1}} T_i \to T_i/T_{i-1} \to 0, \end{equation}
    where $h_{i-1}$ is the minimal right $\add(\bigoplus_{j=1}^{i-1} T_j)$-approximation of $T_i$ of \cref{lem:tiltapprox}(1). Applying $\Hom(P(t),-)$ to \cref{eq:sesexceptional} and using (1), yields
    \[ \dim \Hom(P(t), T_i/T_{i-1}) = \dim \Hom(P(t), T_i) - \dim \Hom(P(t), T_{i-1}) = i-(i-1)=1. \]
    
     To show that each $T_i/T_{i-1}$ is indecomposable and thin and thus complete the proof, it is sufficient to show that $T_i/T_{i-1}$ satisfies $\End(T_i/T_{i-1})\simeq K$ by \cite[Cor. 2.4]{HillePloog2019}.
    Applying $\Hom(T_{i-1},-)$ to \cref{eq:sesexceptional} yields an exact sequence
    \[ \Hom(T_{i-1}, T_{i-1}) \xrightarrow{\Hom(T_{i-1}, h_{i-1})} \Hom(T_{i-1}, T_i) \to \Hom(T_{i-1}, T_i/T_{i-1}) \to \Ext^1(T_{i-1}, T_{i-1})\] 
    which implies that $ \Hom(T_{i-1}, T_i/T_{i-1}) = 0$ because $\Hom(T_{i-1}, h_{i-1})$ is surjective since $h_{i-1}$ is a right approximation and $\Ext^1(T_{i-1}, T_{i-1})=0$ since $T$ is tilting. Consequently, applying $\Hom(-, T_i/T_{i-1})$ to \cref{eq:sesexceptional} gives an exact sequence
    \begin{equation}\label{eq:endoquotientdim} 0 \to \End(T_i/T_{i-1}) \to \Hom(T_i, T_i/T_{i-1}) \to \Hom(T_{i-1},T_i/T_{i-1}) = 0.\end{equation}
    
    There is an equivalence of categories $\Hom(T,-): \add T \to \add \A_t$ because left multiplication defines an isomorphism $\A_t \to \End(T)$ by \cref{thm:tiltingendo}. This implies that
    \[ \dim \Hom(P(i),P(i))=i \quad \text{and} \quad \dim \Hom(P(i), P(i-1))=i-1\]
    restrict to $T_i$ and $T_{i-1}$. Applying $\Hom(T_i, -)$ to \cref{eq:sesexceptional} and using $\Ext^1(T_i, T_{i-1})=0$, because $T$ is tilting, yields
    \[ \dim \Hom(T_i, T_i/T_{i-1})= \dim \End(T_i, T_i)- \dim(T_i, T_{i-1})= i-(i-1) = 1.\]
    Combining this with the isomorphism $\End(T_i/T_{i-1}) \simeq \Hom(T_i, T_i/T_{i-1})$ of \cref{eq:endoquotientdim} yields the desired $\End(T_{i}/T_{i-1}) \simeq K$. This completes the proof that $T_i/T_{i-1}$ is exceptional and thus the proof of the statement. 
\end{proof}

\begin{remark}
    Not every exceptional sequence in $\mods \A_t$ is also a $\tau$-exceptional sequence. This can be immediately observed from the fact that not every exceptional module is $\tau$-rigid. Indeed, consider the case $t=2$ in \cref{exmp:t2}. The module $\begin{smallmatrix} & 2\\ 1& \end{smallmatrix}$, which can equivalently be written as $P(2)/S(1)$, is exceptional but not $\tau$-rigid, as can be easily checked. 
\end{remark}

\begin{remark}\label{rmk:excepnottau}
    For an arbitrary finite dimensional $K$-algebra, not every complete exceptional sequence is necessarily a $\tau$-exceptional sequence. Consider the algebra 
    \[ \Gamma = K(1 \xrightarrow{a} 2 \xrightarrow{b}3 ) /\langle ab \rangle.\]
     It can readily be checked that $(P_2, P_1, S_1)$ is a complete exceptional sequence but not a $\tau$-exceptional sequence because $\Hom(P_2, \tau S_1) \neq 0$.
\end{remark}

\section{Comparing mutations}\label{sec:mutation}
Now that it is established that each complete exceptional sequence in $\mods \A_t$ is a complete $\tau$-exceptional sequence, it is natural to compare their mutation theories. Recall that generally the mutation of complete exceptional sequences is an operation in $\Dcal^b(\mods \A_t)$ whereas the mutation of complete $\tau$-exceptional sequences remains in $\mods \A_t$. Let $(E,F)$ be an exceptional sequence consisting of two terms, also called an exceptional pair. There is a natural map
\[ \Hom(E,F) \otimes_K E \to F, \quad \phi \otimes e \mapsto \phi(e) \]
which, for each integer $\ell$, defines a natural map
\[ \Hom(E[-\ell ],F) \otimes_K E[-\ell] = \Hom(E, F[\ell]) \otimes_K E \to F. \]
Taking direct sums this defines the \textit{canonical map} and a triangle in $\Dcal^b(\mods \A_t)$ given by
\begin{equation}\label{eq:leftmutexcep} \RHom(E,F) \otimes E \to F \to \mathcal{L}_E F \to  \quad .\end{equation}
The \textit{left mutation} of $(E,F)$, later called \textit{left $\psi$-mutation} to distinguish it from the mutation of $\tau$-exceptional sequence, is then defined to be the exceptional pair $(\Lcal_E F, E)$, see \cite[Sec. 2.2]{Gorodentsev1990}. The dual canonical map similarly yields a triangle
\[ \mathcal{R}_F E \to E \to \RHom(E,F)^* \otimes E \to \]
in $\Dcal^b(\mods \A_t)$ which defines the \textit{right $\psi$-mutation} $(F, \Rcal_F E)$ of $(E,F)$. 

\begin{definition}\label{def:completeexcepmut}
    Let $\Ecal = (E_t, \dots, E_1)$ be a complete exceptional sequence in $\mods \A_t$ and let $i \in \{2, \dots, t\}$. Define the \textit{left $\psi$-mutation of $\Ecal$ at (position) $i$} to be
    \[ \psi_i(\Ecal) \coloneqq (E_t, \dots, E_{i+1}, \Lcal_{E_i} E_{i-1}, E_i, E_{i-2},\dots, E_1). \]
    Similarly, define the \textit{right $\psi$-mutation of $\Ecal$ at (position) $i$} to be 
    \[ \psi_i^{-1}(\Ecal) \coloneqq (E_t, \dots, E_{i+1}, E_{i-1}, \Rcal_{E_{i-1}} {E}_i, E_{i-2}, \dots, E_1). \] 
\end{definition}

This notation is justified because these operations are inverse to each other by \cite[Sec. 2.3]{Gorodentsev1990}. It is important to remember that the left and right $\psi$-mutations of a complete exceptional sequence in $\mods \A_t$ may not be complete exceptional sequence in $\mods \A_t$. This property of right $\psi$-mutation is characterised in \cite[Lem. 4.7, Cor. 4.8]{HillePloog2019}. To prove the analogous statement for left $\psi$-mutation and for various other proofs, the following collection of statements about complete exceptional sequences in $\mods \A_t$ is essential.

\begin{lemma}\label{lem:completeexcepprop}
    Let $(E_t, \dots, E_1)$ be a complete exceptional sequence in $\mods \A_t$. The following hold:
    \begin{enumerate}
        \item $\dim \Hom(E_i, E_j) = 1 = \dim \Ext^1(E_i,E_j)$ for any $i>j$;
        \item $\Ext^2(E_i, E_j)=0$ for any $i>j$; 
        \item Every nonzero morphism $E_i \to E_{i-1}$ is either injective or surjective for all $i \in \{2, \dots, t\}$;
        \item If $E_2 \to E_1$ is injective, then $\projdim E_2 \leq 1$;
        \item If $f:E_i \to E_{i-1}$ is injective for $i \in \{2, \dots, t\}$, then $\dim \Ext^1(\coker f, E_i)=1$.
    \end{enumerate} 
\end{lemma}
\begin{proof}
    (1) and (2) are \cite[Cor. 3.7]{HillePloog2019}. (3) holds for worm diagrams by \cite[Lem. 3.2]{HillePloog2019}, which are equivalent to complete exceptional sequences by \cite[Prop. 3.3, Cor. 3.6]{HillePloog2019}. (4) The exceptional module $E_1$ is a submodule of $P(1)$, which is uniserial and thin with socle $S(t)$. If $E_2 \to E_1$ is injective, then $E_2$ is also uniserial and thin with socle $S(t)$. Notice that $\topp(E_2) \not \simeq S(1)$, as otherwise $E_2 \simeq P(1)$ leads to a contradiction. Let $\topp(E_2) \simeq S(j)$ for some $j \in \{2, \dots, t\}$. It is easily checked that its minimal projective resolution is given by 
    \[ 0 \to P(j-1) \to P(j) \to E_2 \to 0. \] 
    (5) Applying $\Hom(-,E_i)$ to the short exact sequence defining $\coker f$ yields an exact sequence
    \begin{equation}\label{eq:cseqpropcoker} 0 = \Hom(E_{i-1},E_i) \to \Hom(E_i,E_i) \to \Ext^1(\coker f, E_i) \to \Ext^1(E_{i-1},E_i) = 0\end{equation}
    where the outermost terms are zero by the defining property of exceptional sequences. Since $E_i$ is exceptional, the equality $\dim \Hom(E_i, E_i)=1$ holds and the result follows from the isomorphism $\Hom(E_i, E_i) \simeq \Ext^1(\coker f, E_i)$ in \cref{eq:cseqpropcoker}.
\end{proof}

The dual of \cite[Lem. 4.7, Cor. 4.8]{HillePloog2019} is then readily proved as follows.
\begin{lemma}\label{lem:excepmutdesc}
    Let $\Ecal = (E_t, \dots, E_1)$ be a complete exceptional sequence in $\mods \A_t$ and let $i \in \{2, \dots, t\}$. The following are equivalent:
    \begin{enumerate}
        \item any nonzero morphism $f: E_{i} \to E_{i-1}$ is injective;
        \item $\psi_i(\Ecal)$ is a complete exceptional sequence in $\mods \A_t$.
    \end{enumerate} 
    In this case $\Lcal_{E_i} E_{i-1} \in \mods \A_t$ is defined by the short exact sequence
    \begin{equation}\label{eq:leftmutseq} 0 \to \coker f \to \Lcal_{E_i} E_{i-1} \to E_i \to 0\end{equation}
    in $\mods \A_t$.
\end{lemma}
\begin{proof}
    This statement is unambiguous because $\dim \Hom(E_{i}, E_{i-1})=1$ by \cref{lem:completeexcepprop}(1) and therefore $\coker f_1 \simeq \coker f_2$ for any two $f_1, f_2 \in \Hom(E_i, E_{i-1})$. 

    The triangle defining $\Lcal_{E_i} E_{i-1}$, given by \cref{eq:leftmutexcep}, gives rise to a long exact sequence in cohomology
    \begin{equation}\begin{aligned}\label{eq:longexactcoho} 0 = H^{-1}(E_{i-1}) \to H^{-1}(\Lcal_{E_i} E_{i-1}) \to \Hom(E_i, E_{i-1})\otimes E_i \to E_{i-1} \to H^0(\Lcal_{E_i} E_{i-1}) \\
    \to \Ext^1(E_i, E_{i-1}) \otimes E_i \to H^1(E_{i-1}) = 0
    \end{aligned}\end{equation}
    using the fact that $E_{i-1}$ is concentrated in degree zero in $\Dcal^b(\mods \A_t)$ to get $H^{\neq 0}(E_{i-1})= 0$; the fact that $\Ext^2(E_i, E_{i-1})= 0$ by \cref{lem:completeexcepprop}(2) and the fact that $\Ext^{\geq 3}(E_i, E_{i-1}) = 0$ since $\A_t$ is of global dimension 2 to conclude that $\Lcal_{E_i} E_{i-1}$ is concentrated degrees -1 and 0. 

    Moreover, since $\dim \Hom(E_i, E_{i-1})=1 = \Ext^1(E_i, E_{i-1})$ by \cref{lem:completeexcepprop}(1), the tensor products in \cref{eq:longexactcoho} simplify to yield the exact sequence
    \begin{equation}\label{eq:simplifiedexcepdesc} 0 \to H^{-1}(\Lcal_{E_i} E_{i-1}) \to E_i \to E_{i-1} \to H^0(\Lcal_{E_i} E_{i-1}) \to E_i \to 0. \end{equation}

    Note that $E_i \to E_{i-1}$ is nonzero in \cref{eq:simplifiedexcepdesc}, so that \cref{lem:completeexcepprop}(3) implies that it is either surjective or injective. If it is injective it follows that $H^{-1}(\Lcal_{E_i} E_{i-1})=0$ and $\Lcal_{E_i} E_{i-1} \in \mods \A_t$ has the desired description, and implying that $\psi_i(\Ecal)$ is a complete exceptional sequence in $\mods \A_t$. 

    Conversely, if $\psi_i(\Ecal)$ is a complete exceptional sequence in $\mods \A_t$, then in particular $\Lcal_{E_i} E_{i-1} \in \mods \A_t$, implying $H^{-1}(\Lcal_{E_i} E_{i-1}) = 0$, consequently $f: E_i \to E_{i-1}$ is injective. This concludes the proof.
\end{proof}

In view of the previous proposition and its dual, see \cite[Lem. 4.7, Cor. 4.8]{HillePloog2019}, the following short-hand notation will be used throughout.

\begin{definition}\label{def:psimut}
    Let $\Ecal= (E_t, \dots, E_1)$ be a complete exceptional sequence in $\mods \A_t$ and let $i \in \{2, \dots, t\}$. Say that 
    \begin{itemize}
        \item $\Ecal$ can be left $\psi$-mutated at (position) $i$ if any nonzero $E_{i} \to E_{i-1}$ is injective.
        \item $\Ecal$ can be right $\psi$-mutated at (position) $i$ if any nonzero $E_{i} \to E_{i-1}$ is surjective. 
    \end{itemize}
    The corresponding $\psi$-mutation at (position) $i$ is as defined in \cref{def:completeexcepmut}.
\end{definition}

\subsection{Mutation of $\tau$-exceptional sequences}
Before discussing the mutation of $\tau$-exceptional sequences and comparing it to the mutation of exceptional sequences further notation is required. Let $\Ccal \subseteq \mods \A_t$ be a full subcategory and let $\Tcal$ be a torsion class in $\mods \A_t$. Denote by $\Trm(\Ccal)$ the smallest torsion class of $\mods \A_t$ containing $\Ccal$ and denote by $\Prm(\Tcal) \in \mods \A_t$ the basic module isomorphic to a direct sum all indecomposable $\Ext$-projective objects in $\Tcal$. Similarly, denote by $\Prm_{\ssplit}(\Ccal)$ and $\Prm_{\nsplit}(\Ccal)$ the maximal direct summands of $\Prm(\Tcal)$ consisting of split and nonsplit $\Ext$-projective objects respectively.

The theory of mutation of $\tau$-exceptional sequences was introduced in \cite{BHM2024}. Similar to the mutation of exceptional sequences, mutation is an operation on $\tau$-exceptional pairs, that is, on $\tau$-exceptional sequence of length 2. Left and right mutations of $\tau$-exceptional pairs in $\mods \A_t$ are divided into two cases. 

\begin{definition}\label{defn:leftrightregular} \cite[Def. 3.11]{BHM2024}
    A $\tau$-exceptional pair $(B,C)$ in $\mods \A_t$ is called \textit{left regular} if $C$ is projective in $\mods \A_t$ or if $C \not \in \Prm( {}^\perp \tau (f_C^{-1}(B)))$. Otherwise, it is called \textit{left irregular}.
    
    A $\tau$-exceptional pair $(B,C)$ in $\mods \A_t$ is called \textit{right regular} if $f_C^{-1}(B) \in \Prm({}^\perp \tau C)$ or if $C \not \in \Gen(f_C^{-1}(B))$. Otherwise, it is called \textit{right irregular}.
\end{definition}

It is important to remark that over a general finite dimensional algebra $\Lambda$, there may exist (left or right irregular) $\tau$-exceptional sequences which are (left or right) immutable. However, the algebra $\A_t$ has only finitely many isomorphism classes of basic $\tau$-tilting modules by \cite[Thm. 1.2]{IyamaZhang2020}. Therefore, every $\tau$-exceptional pair of $\mods \A_t$ can be left and right mutated by \cite[Thm. 0.2]{BHM2024}. To describe this process, the following bijection extending that of \cref{rmk:notationjustification} is important, see \cite[Prop. 5.6]{BuanMarsh2018t} and \cite[Sec. 3]{BuanHanson2023}. Although the precise expression of this bijection is not used in the remaining proofs of this paper, it is included for the sake of completeness.

\begin{theorem}\cite[Thm. 2.16]{BHM2024}
	Let $(T,P)$ be a $\tau$-rigid pair in $\mods \A_t$. Then there is a bijection
	\[ \begin{tikzcd}
	 \{ \text{indecomposable } M \in \mods \A_t \sqcup \mods \A_t[1]: M \oplus T \oplus P[1] \text{ is a $\tau$-rigid pair in $\mods \A_t$} \} \arrow[d, "\Erm_{T \oplus P[1] }"] \\
	 \{ \text{indecomposable } N \in J(T,P) \sqcup J(T,P)[1]: N \text{ is a $\tau$-rigid pair in $J(T,P)$} \}
	 \end{tikzcd}
	 \]
	 given by
	 \[ \Erm_{T \oplus P[1]}(M) = \begin{cases}  \Prm_{\ssplit}(J(T,P) \cap {}^\perp J(T \oplus M \oplus P[1])) & \text{ if $M \in \mods \A_t$ or $M \not \in \Gen T$;}\\
	 \Prm_{\ssplit}(J(T,P) \cap {}^\perp J(T \oplus M \oplus P[1]))[1] & \text{ otherwise.}  \end{cases} \]
\end{theorem}

Let $(B,C)$ be a $\tau$-exceptional pair in $\mods \A_t$, the left and right mutations are introduced following \cite[Sec. 4]{BHM2024}. Assume that $(B,C)$ is left regular. Then the \textit{left mutation}, also called \textit{left $\varphi$-mutation} to distinguish it from the $\psi$-mutation of exceptional sequences, of $(B,C)$ is the $\tau$-exceptional pair $\varphi(B,C) \coloneqq (\widehat{C}, \widehat{B})$, where 
\[ \widehat{B} = \begin{cases} B & \text{ if $C$ is projective;} \\ f_C^{-1}(B) & \text{ otherwise;} \end{cases} \quad \text{and} \quad \widehat{C} = \begin{cases} \lfloor \Erm_{B}(C[1]) \rfloor & \text{ if $C$ is projective;} \\ \lfloor \Erm_{\widehat{B}}(C) \rfloor & \text{ otherwise.} \end{cases} \]

In the expression for $\widehat{C}$, the objects $\Erm_B(C[1])$ and $\Erm_{\widehat{B}}(C)$ may be $\tau$-rigid pairs of the form $(0,P)$ for some projective $P\in \mods \A_t$ and applying $\lfloor - \rfloor$ means considering them as $\tau$-rigid modules $P \in \mods \A_t$. Note that the description of $\widehat{C}$ will not be necessary for proving the main results in this paper.

Assume now that $(B,C)$ is left irregular and define 
\[ X = \Prm_{\ssplit}(\Gen \Prm_{\nsplit}(\Trm(J(f_C^{-1}(B) \oplus C)))) \quad \text{and} \quad Y = \Prm_{\nsplit}(\Trm(J(f_C^{-1}(B) \oplus C)))/X.\]
The left $\varphi$-mutation of the $\tau$-exceptional pair $(B,C)$ is given by $\varphi(B,C) \coloneqq (f_Y(X), Y)$ in this case.

Now consider the right mutation of $\tau$-exceptional sequences. By \cite[Thm. 4.7]{BHM2024}, right mutation is inverse to left mutation. Assume that $(B,C)$ is right regular. The \textit{right ($\varphi$-)mutation} of $(B,C)$ is defined as $\varphi^{-1}(B,C) = (\overline{C}, \overline{B})$, where
\[ \overline{C} = \begin{cases} \Erm_{\Erm_C^{-1}(B[1])}(C) & \text{ if $B$ is projective in $J(C)$;} \\ \Erm_{\overline{B}}(C) & \text{ otherwise;} \end{cases} \quad \text{and} \quad \overline{B} = \begin{cases} |\Erm_C^{-1}(B[1])| & \text{ if $B$ is projective in $J(C)$;} \\ f_C^{-1}(B) & \text{otherwise.}\end{cases}\]

For the algebra $\A_t$, the first case in the definition of $\overline{C}$ can be drastically simplified with the help of the following.
\begin{lemma}\label{lem:rightmuthelp}
	Let $(E_t, \dots, E_1)$ be a complete exceptional sequence in $\mods \A_t$ such that $E_2$ is projective in $J(E_1)$. Then $\varphi^{-1}(E_2, E_1) = (E_1, ?)$ for some $? \in \mods A_t$.
\end{lemma}
\begin{proof}
	First, consider the expression $X\coloneqq \Erm_{E_1}^{-1}(E_2[1])$ as it appears in the subscript of the definition of $\overline{E_1}$ above. By definition of the E-map, see for example \cite[Def. 3.5]{BKT2025} or the proof of \cite[Def.-Prop. 4.5]{BHM2024}, there are two cases to consider. The first case is $X= P[1]$ for some projective $P \in \mods \A_t$. In this case the definition of the E-map simply yields $\Erm_X(E_1)=E_1$ as required. 
	
	The second case is $X \in \Gen E_1$. It follows then that $E_1 \not \in \Gen X$ by \cite[Lem. 1.12]{BHM2024}. The definition of the E-map gives that $\Erm_X(E_1) = f_X (E_1)$. Assume for a contradiction that $\Hom(X, E_1) \neq 0$. Because $X \in \Gen E_1$, there then exists an endomorphism of $E_1$ which is not an isomorphism. This is a contradiction to $E_1$ being exceptional. Consequently $\Hom(X, E_1) = 0$ from which $f_{X} (E_1) = E_1$, and thus the desired result, follows.
\end{proof}

The definition of right $\varphi$-mutation for right irregular $\tau$-exceptional pairs is omitted because it will not be required to prove the main results herein.

\begin{definition}\label{def:completetauexcepmut}
    Let $\mathcal{M} = (M_t, \dots, M_1)$ be a complete $\tau$-exceptional sequence. The \textit{left $\varphi$-mutation of $\mathcal{M}$ at (position) $i \in \{2, \dots, t\}$} is defined to be $\varphi_i(\mathcal{M}) = (M_t, \dots, M_{i+1}, \widehat{M}_{i-1}, \widehat{M}_i, M_{i-2}, \dots, M_1)$, where
    \[ (\widehat{M}_{i-1}, \widehat{M}_{i}) = \varphi^{J(M_{i-2}, \dots, M_1)}(M_i, M_{i-1})\]
    is the left $\varphi$-mutation of the $\tau$-exceptional pair in the subcategory $J(M_{i-2}, \dots, , M_1)$.
    The \textit{right $\varphi$-mutation of $\mathcal{M}$ at (position) $i \in \{2, \dots, t\}$} is defined to be $\varphi_i^{-1}(\mathcal{M}) = (M_t, \dots, M_{i+1}, \overline{M}_{i-1}, \overline{M}_i, M_{i-2}, \dots, M_1)$, where
    \[(\overline{M}_{i-1}, \overline{M}_i) = (\varphi^{-1})^{J(M_{i-2}, \dots, M_1)}(M_i, M_{i-1}).\]
\end{definition}

The following theorem justifies that the mutation of $\tau$-exceptional sequences described in this subsection is well-defined.
\begin{theorem}\cite[Thm. 0.1, Cor. 0.4]{BHM2024}\label{thm:BHMmut}
    The left $\varphi$-mutation of a complete $\tau$-exceptional sequence in $\mods \A_t$ is again a complete $\tau$-exceptional sequence in $\mods \A_t$. Similarly, the right $\varphi$-mutation of a complete $\tau$-exceptional sequence in $\mods \A_t$ is again a complete $\tau$-exceptional sequence in $\mods \A_t$. Moreover right $\varphi$-mutation is inverse to left $\varphi$-mutation and vice versa.
\end{theorem}

By \cref{thm:excepistauexcep}, every complete exceptional sequence is also a complete $\tau$-exceptional sequence. For this reason, it will always be made explicit whether $\psi$-mutation (of exceptional sequences) or $\varphi$-mutation (of $\tau$-exceptional sequences) is considered.

Notice that the only projective exceptional $\A_t$-module is $P(1)$. Thus, the following lemma shows that most exceptional pairs which arise in a complete exceptional sequence and which can be left $\psi$-mutated are left irregular $\tau$-exceptional pairs. Nonetheless, in this case left $\varphi$-mutation coincides with left $\psi$-mutation. 

\begin{lemma}\label{lem:leftregmut}
    Let $\Ecal= (E_t, \dots, E_1)$ be a complete exceptional sequence in $\mods \A_t$ which can be left $\psi$-mutated at position $2$. Assume that $(E_2, E_1)$ is a left regular $\tau$-exceptional pair. Then $E_1$ is projective in $\mods \A_t$ and $\varphi_{2}(\Ecal) = \psi_{2}(\Ecal)$.
\end{lemma}
\begin{proof}
    Since $(E_2, E_1)$ is left regular, \cref{defn:leftrightregular} states that either $E_1$ is projective in $\mods \A_t$ or $E_1 \not \in \Prm({}^\perp \tau (f_{E_1}^{-1}(E_2)))$. It is now shown that $E_1$ is projective in ${}^\perp \tau (f_{E_1}^{-1}(E_2))$.

    Write $X= f_{E_1}^{-1}(E_2)$ and let $Y \in {}^\perp \tau X$, so that $\Hom(Y, \tau X)=0$. This implies that $\Ext^1(X,Y) = 0$ by \cite[Cor. IV.4.7]{ARS1995}. By \cref{thm:excepistauexcep}, the exceptional sequence $\Ecal$ comes from a tilting module $T$ such that $T_2 \oplus T_1 = X \oplus E_1$. Thus, \cref{lem:torfdesc} implies that $E_1$ is a submodule of $X$ and yields a short exact sequence 
    \[ 0 \to E_1 \to X \to E_2 \to 0. \]
    Applying $\Hom(-,Y)$ to the above yields an exact sequence
    \[ 0 = \Ext^1(X,Y) \to \Ext^1(E_1,Y) \to \Ext^2(E_2, Y),\]
    where $\Ext^2(E_2, Y) = 0$ because $\projdim E_2 \leq 1$ by \cref{lem:completeexcepprop}(4) and the assumption that $(E_2, E_1)$ can be left $\psi$-mutated. It follows that $\Ext^1(E_1, Y) = 0$. Thus $E_1$ is projective in ${}^\perp \tau (f_{E_1}^{-1}(E_2))$. Consequently $(E_2, E_1)$ is a left regular $\tau$-exceptional pair if and only if $E_1$ is projective in $\mods \A_t$.

    Since $E_1$ is projective in $\mods \A_t$, the formulas for left mutations give
    \[ \psi_{2}(\Ecal) = (E_t, \dots, E_3, ?, E_2) \quad \text{and} \quad \varphi_{2}(\Ecal) = (E_t, \dots, E_3, !,E_2).\]

    By \cref{lem:excepmutdesc}, the sequence $\psi_{2}(\Ecal)$ is a complete exceptional sequence in $\mods \A_t$ and hence also a complete $\tau$-exceptional sequence by \cref{thm:excepistauexcep}. Similarly, $\varphi_{2}(\Ecal)$ is a complete $\tau$-exceptional sequence by \cref{thm:BHMmut}. However, two complete $\tau$-exceptional sequences which differ in at most term are equal by \cite[Thm. 2.1(a)]{BuanMarsh2023}, see also \cite[Thm. 8]{HansonThomas2024}. This concludes the proof.
\end{proof}

The following result similarly describes right regular $\varphi$-mutation. 

\begin{lemma}\label{lem:rightirrcounter}
    Let $\Ecal= (E_t, \dots, E_1)$ be complete exceptional sequence which can be right $\psi$-mutated at position 2. Assume that $(E_2, E_1)$ is a right regular $\tau$-exceptional pair in $\mods \A_t$. Then $f_{E_1}^{-1}(E_2)$ is projective in ${}^\perp \tau E_1$ and $\varphi_{2}^{-1}(\Ecal) = \psi_{2}^{-1}(\Ecal)$.
\end{lemma}
\begin{proof}
    Since $(E_2, E_1)$ is right regular, \cref{defn:leftrightregular} states that either $E_1 \not \in \Gen(f_{E_1}^{-1}(E_2))$ holds or $f_{E_1}^{-1}(E_2) \in \Prm({}^\perp \tau E_1)$ holds. Clearly $E_2 \in \Gen (f_{E_1}^{-1}(E_2))$ holds. By assumption $(E_2, E_1)$ can be right $\psi$-mutated, so there is a surjective morphism $E_2 \to E_1$, whence $E_1 \in \Gen E_2$ is satisfied. In conclusion, $E_1 \in \Gen (f_{E_1}^{-1}(E_2))$, proving the first part of the statement.

    The fact that $f_{E_1}^{-1}(E_2)$ is projective in ${}^\perp \tau E_1$ implies that $E_2$ is projective in $J(E_1)$, see \cite[Lem. 2.17]{BHM2024}. In this case it follows from \cref{lem:rightmuthelp} that
    \[ \psi_{2}^{-1}(\Ecal) = (E_t, \dots, E_3, E_1, ?) \quad \text{and} \quad \varphi_{2}^{-1}(\Ecal) = (E_t, \dots, E_3,E_1, !).\]

    By \cite[Lem. 4.7, Cor. 4.8]{HillePloog2019}, the sequence $\psi_{2}^{-1}(\Ecal)$ is a complete exceptional sequence and hence also a complete $\tau$-exceptional sequence by \cref{thm:excepistauexcep}. Similarly, $\varphi_{2}^{-1}(\Ecal)$ is a complete $\tau$-exceptional sequence by \cref{thm:BHMmut}. However, two complete $\tau$-exceptional sequences which differ in at most term are equal by \cite[Thm. 2.1(a)]{BuanMarsh2023}, see also \cite[Thm. 8]{HansonThomas2024}. This concludes the proof.
\end{proof}

Knowing that $\varphi$-mutation coincides with $\psi$-mutation whenever the right-most exceptional pair is a regular $\tau$-exceptional pair, it remains to consider the irregular case. The following proposition collects the necessary results to show a similar result for the irregular pairs. 

\begin{proposition}\label{prop:leftirregularprops}
    Let $\Ecal= (E_t, \dots, E_1)$ be complete exceptional sequence in $\mods \A_t$ which can be left $\psi$-mutated at position $2$. Recall the notation $\psi(E_2,E_1) = (\Lcal_{E_2} E_1, E_2)$. Let $T = T_t \oplus \dots \oplus T_1$ be the TF-ordered tilting $\A_t$-module in bijection with $\Ecal$ via \cref{thm:excepistauexcep}. The following hold:
    \begin{enumerate}
        \item $T_2 \oplus E_1$ and $T_2 \oplus E_2$ are partial tilting modules;
        \item  There is an isomorphism $f_{E_2}^{-1}(\Lcal_{E_2} E_1) \simeq T_2$;
        \item $(\Lcal_{E_2} E_1, E_2)$ is an exceptional pair in $\mods \A_t$ which can be right $\psi$-mutated;
        \item $J(T_2 \oplus E_1) = J(T_2 \oplus E_2)$.
    \end{enumerate}
    Moreover, assume that $(E_2, E_1)$ is a left irregular $\tau$-exceptional pair in $\mods \A_t$, then 
    \begin{enumerate}
    	\item[(5)] $T_2$ is not projective;
        \item[(6)] $(\Lcal_{E_2} E_1, E_2)$ is a right irregular $\tau$-exceptional pair.
    \end{enumerate}
\end{proposition}
\begin{proof}
    (1) The fact that $E_1 \oplus T_2$ is partial tilting follows immediately from the equality $E_1=T_1$ implied by \cref{thm:excepistauexcep}. By definition, the universal coextension $U$ of $E_2$ by $E_1$ is determined by the short exact sequence
    \[ 0 \to E_1 \to U \to \Ext^1(E_2, E_1) \otimes_K E_2 \to 0. \]
    Note that since $\dim \Ext^1(E_2, E_1) = 1$ by \cref{lem:completeexcepprop}(1), the module $T_2$ is isomorphic to $U$. The result \cite[Lem. 3.2(a)]{HillePerling2014} then states that $\Ext^1(E_2, T_2)= 0 = \Ext^j(T_2, E_2)$ for $j \geq 1$. 
    
    Since $E_2$ is exceptional and $T_2$ is partial tilting, it follows that also $\Ext^1(E_2, E_2) = 0 = \Ext^1(T_2, T_2) $ holds. In conclusion $\Ext^1(E_2 \oplus T_2, E_2 \oplus E_2) = 0$. By \cref{lem:completeexcepprop}(4), the inequality $\projdim(E_2) \leq 1$ holds, so that $E_2 \oplus T_2$ is partial tilting as desired.

    (2) By assumption  $\psi_{2}(\Ecal) = (E_t, \dots, E_3, \Lcal_{E_2} E_1, E_2)$ is a complete exceptional sequence in $\mods \A_t$. By \cref{thm:excepistauexcep}, it is also a complete $\tau$-exceptional sequence so that the expression $f_{E_2}^{-1}(\Lcal_{E_2} E_1)$ is well-defined, see \cref{rmk:notationjustification}. Let $f: E_2 \to E_1$ be a nonzero injective morphism, which exists by assumption, see \cref{def:psimut}. Consider the short exact sequence of \cref{eq:leftmutseq} and a nonsplit short exact sequence in $\Ext^1(\Lcal_{E_2} E_1, E_1)$, which exists and is unique up to isomorphism by \cref{lem:completeexcepprop}(1).  These short exact sequences are highlighted in {\color{purple}\textbf{purple}} in the following commutative diagram.
    \[
    \begin{tikzcd}
        &&0 \arrow[d] & {\color{purple}\boldsymbol{0}} \arrow[d, purple] \\
        0 \arrow[r] & E_2 \arrow[d, equal] \arrow[r] & Y \arrow[r, twoheadrightarrow] \arrow[d, hookrightarrow] \arrow[rd, phantom, "\ulcorner", very near end] & {\color{purple}\boldsymbol{\coker f}} \arrow[d, hookrightarrow, purple] \arrow[r] & 0\\
        {\color{purple}\boldsymbol{0}} \arrow[r, purple] & {\color{purple}\boldsymbol{E_2}} \arrow[r, hookrightarrow, purple] & {\color{purple}\boldsymbol{Z}} \arrow[r, twoheadrightarrow, purple] \arrow[d, twoheadrightarrow] & {\color{purple}\boldsymbol{\Lcal_{E_2} E_1}} \arrow[r, purple] \arrow[d, twoheadrightarrow, purple] & {\color{purple}\boldsymbol{0}} \\
        &&E_2 \arrow[r, equal] \arrow[d] &{\color{purple}\boldsymbol{E_2}} \arrow[d, purple] \\
        &&0 &{\color{purple}\boldsymbol{0}}
    \end{tikzcd}
    \]
    This diagram is constructed via a pullback in the {\color{purple}\textbf{purple}} subdiagram as indicated. The pullback yields the exactness and commutativity of the rows in the diagram \cite[Prop. I.13.1]{Mitchell}. Since $\coker f \hookrightarrow \Lcal_{E_2} E_1$ is a monomorphism, so is  $Y \to Z$ by \cite[Prop. I.7.1]{Mitchell}. Since $\mods \A_t$ is abelian, the morphism $Z \twoheadrightarrow \Lcal_{E_2} E_1$ being an epimorphism implies that $Y \to \coker f$ is an epimorphism \cite[Prop. I.20.2]{Mitchell}. Moreover, the pullback square is also a pushout square by \cite[p. 39, Ex. I.17]{Mitchell}. The exactness and commutativity of the columns therefore hold by the dual of \cite[Prop. I.7.1]{Mitchell}.
    
    By \cref{lem:completeexcepprop}(5), $\dim \Ext^1(\coker f, E_2) = 1$, from which $Y \simeq E_1$ follows. By \cref{lem:completeexcepprop}(1), $\dim \Ext^1(E_2, E_1)=1$, from which $Z \simeq T_2$ follows. The module $T_2 \oplus E_2$ is $\tau$-rigid by (1) and \cref{lem:tiltistau}, which clearly makes it a TF-ordered $\tau$-rigid module in $\mods \A_t$. Since it maps to $(\Lcal_{E_2} E_1, E_2)$ under the bijection $\Phi$, the desired isomorphism $f_{E_2}^{-1}(\Lcal_{E_2} E_1) \simeq T_2$ follows from \cref{thm:MendozaTreff}.

    (3) By \cref{lem:excepmutdesc}, there exists a nonzero surjection $\Lcal_{E_2} E_1 \to E_2$. Hence all nonzero morphisms from $\Lcal_{E_2} E_1$ to $E_2$ are surjective by \cref{lem:completeexcepprop}(3).

    (4) By (1) the modules $T_2 \oplus E_1$ and $T_2 \oplus E_2$ are both partial tilting, in particular of projective dimension at most 1. Then \cite[Cor. IV.4.7]{ARS1995} implies that 
    \[ J(T_2 \oplus E_i) = \{ X \in \mods \A_t : \Hom(T_2 \oplus E_i, X) = 0 = \Ext^1(T_2 \oplus E_i,X) \} \]
    for $i=1,2$. The result follows from the long exact sequence obtained by applying $\Hom(-,X)$ to the short exact sequence
    \[ 0 \to E_1 \to T_2 \to E_2 \to 0 \] 
    induced by \cref{thm:excepistauexcep}.
    
    (5) This is immediate by \cite[Prop. 3.13(a)]{BHM2024} since $f_{E_1}^{-1}(E_2) = T_2$.

    (6) By (3) the pair $(\Lcal_{E_2} E_1, E_2)$ is an exceptional pair in $\mods \A_t$ which can be right $\psi$-mutated. The complete exceptional sequence $\psi_{2}(\Ecal)$ is also a complete $\tau$-exceptional sequence by \cref{thm:excepistauexcep}. This implies that $(\Lcal_{E_2} E_1, E_2)$ is a $\tau$-exceptional pair. 

    Assume for a contradiction that $(\Lcal_{E_2} E_1, E_2)$ is a right regular $\tau$-exceptional pair. In this case \cref{lem:rightirrcounter} gives
    \[ \varphi_2^{-1}(\psi_2(\Ecal)) = \psi_{2}^{-1}(\psi_{2}(\Ecal)) = \Ecal. \]
    
    As $(\Lcal_{E_2} E_1, E_2)$ was right regular, \cite[Def.-Prop. 4.5]{BHM2024} implies that $\varphi_{2}^{-1}(\psi_{2}(\Ecal)) = \Ecal$ is left regular at position $2$, that is, $(E_2, E_1)$ is a left regular $\tau$-exceptional pair. This is a contradiction. 
\end{proof}

It is now possible to prove the base case for the inductive proof that the mutations coincide.

\begin{proposition}\label{prop:leftphimut}
    Let $\Ecal = (E_t, \dots, E_1)$ be a complete exceptional sequence in $\mods \A_t$. If $\Ecal$ can be left $\psi$-mutated at position $2$, then $\psi_{2}(\Ecal) = \varphi_{2}(\Ecal)$.
\end{proposition}
\begin{proof}
    This is well-defined because $\Ecal$ is also a complete $\tau$-exceptional sequence by \cref{thm:excepistauexcep}. If $(E_2, E_1)$ is a left regular $\tau$-exceptional pair, the result follows from \cref{lem:leftregmut}. 

    If $(E_2, E_1)$ is a left irregular $\tau$-exceptional pair consider $\psi(E_2, E_1) = (\Lcal_{E_2} E_1, E_2)$. By \cref{prop:leftirregularprops}(6) this is a right irregular $\tau$-exceptional pair. Combining \cite[Prop. 3.14(e)]{BHM2024} with \cref{prop:leftirregularprops}(2) yields
    \[ \tau(T_2 \oplus E_2) = \tau \Prm_{\nsplit}(\Trm(J(T_2 \oplus E_2))). \]
    Observe that $T_2$ and $E_2$ are not projective in $\mods \A_t$ by \cref{prop:leftirregularprops}(5). Using \cref{prop:leftirregularprops}(4) and applying $\tau^{-1}$ yields
    \[ T_2 \oplus E_2 = \Prm_{\nsplit}(\Trm(J(T_2 \oplus E_1))). \]
    Since $E_2 \in \Gen T_2$ it is clear that $T_2 = \Prm_{\ssplit}(\Gen \Prm_{\nsplit}(\Trm(J(T_2 \oplus E_1))))$ and thus $\varphi(E_2, E_1) = (f_{E_2}(T_2), E_2)$ by the definition of left irregular $\varphi$-mutation. It follows from \cref{prop:leftirregularprops}(2) that $f_{E_2}(T_2) \simeq \Lcal_{E_2} E_2$, whence 
    \begin{equation}\label{eq:mutationexpression2} \varphi_{2}(\Ecal) = (E_t, \dots, E_3, \Lcal_{E_2} E_1, E_2) = \psi_{2}(\Ecal) \end{equation}
    holds. 
\end{proof}

This shows that $\psi$-mutation and $\varphi$-mutation at the right-most position of a complete exceptional sequence coincide. To show the same for mutation at other positions, the following reduction lemma is useful.

\begin{lemma}\label{lem:reduction}
    Let $\Ecal = (E_t, \dots, E_1)$ be a complete exceptional sequence in $\mods \A_t$. Then $J(E_i, \dots, E_1)$ is equivalent to $\mods \A_{t-i}$ for all $i \in \{1, \dots, {t-1}\}$. 
\end{lemma}
\begin{proof}
    It is sufficient to show this for $i=1$ because of the recursive definition of $J(E_i, \dots, E_1)$, see \cite[Thm. 1.4]{BuanMarsh2021w} and \cite[Thm. 6.4]{BuanHanson2023}, and the observation that $(E_t, \dots, E_2)$ is a complete exceptional sequence in $J(E_1)$ because the subcategory is full and extension-closed. Indeed, if $J(E_i, \dots, E_1) \simeq \mods \A_{t-i}$ and the case $i=1$ hold, then the proof follows inductively from 
    \[ J(E_{i+1}, E_i, \dots, E_1) = J^{J(E_i, \dots, E_1)}(E_{i+1}) \simeq J^{\mods \A_{t-i}}(E_{i+1}) = \mods \A_{t-i-1}\]
    as required. Let $T$ be the tilting $\A_t$-module such that $\Phi(T) = \Ecal$ via \cref{thm:excepistauexcep}, which is such that $E_1 = T_1$. Consider the $\A_t$-module $T' = \Prm(^\perp \tau T_1)$, which is tilting since $T_1$ is partial tilting by \cite[Thm. A]{LiZhang2015}. It is clear that $e_1(T') = e_1 T$, giving $(T')_1 = T_1$ in the quasi-hereditary decomposition of $T'$. 
    
    By \cref{thm:tiltingendo} the endomorphism algebra of $T'$ satisfies $\End(T') \simeq \A_t$ and by \cref{lem:tiltapprox}(1), the module $T_1$ has a monomorphism to every other direct summand of $T'$. This means that $\Hom(T',T_1)$ is the indecomposable projective $\End(T')$-module which is a submodule of every other indecomposable projective. In other words, $\Hom(T',T_1)$ is isomorphic to the projective at vertex 1 in the presentation of $\A_t$ given at the start of \cref{sec:prelim}. From this presentation it is clear that $\A_t/\langle e_1 \rangle \simeq \A_{t-1}$. It follows that
    \[ J(E_1) \simeq \mods (\End(T')/[\Hom(T', T_1)]) \simeq \mods \A_t / \langle e_1 \rangle \simeq \mods \A_{t-1}, \]
    where the first isomorphism is \cite[Thm. 3.8]{Jasso2015}. This concludes the proof.
\end{proof}

The preceding lemma enables the inductive step in the following proof.

\begin{theorem}\label{thm:mutthm}
    Let $\Ecal = (E_t, \dots, E_1)$ be an exceptional sequence in $\mods \A_t$ and let $i \in \{2, \dots, t\}$. If $\Ecal$ can be left $\psi$-mutated at position $i$, then $\psi_i(\Ecal) = \varphi_i(\Ecal)$.
\end{theorem}
\begin{proof}
    This is well-defined since $\Ecal$ is also a complete $\tau$-exceptional sequence by \cref{thm:excepistauexcep}. If $i=2$, then the result is \cref{prop:leftphimut}. If $i \in \{3, \dots, t\}$, then $J(E_{i-2}, \dots, E_1) \simeq \mods \A_{t-i+2}$ by \cref{lem:reduction}. It follows that 
    \begin{align*} \varphi_i(\Ecal) &= (E_t, \dots, E_{i+1}, \varphi^{J(E_{i-2}, \dots, E_1)}(E_i, E_{i-1}), E_{i-2}, \dots, E_1) \\
    &= (E_t, \dots, E_{i+1}, \psi^{J(E_{i-2}, \dots, E_1)}(E_i, E_{i-1}), E_{i-2}, \dots, E_1) \\
    &= (E_t, \dots, E_{i+1}, \psi(E_i, E_{i-1}), E_{i-2}, \dots, E_1) \\
    & = \psi_i(\Ecal). \end{align*}
    where the first and final equalities hold by definition, the second equality is \cref{prop:leftphimut} in the subcategory $J(E_{i-2}, \dots, E_1) \simeq \mods \A_{t-i+2}$. The third equality follows from the fact that $J(E_{i-2}, \dots, E_1)$ is a wide subcategory of $\mods \A_t$ by \cite[Cor. 3.25]{BST2019}, that is, it is closed under kernels, cokernels and extensions. Therefore, the short exact sequence of \cref{eq:leftmutseq} lies in this subcategory. This concludes the proof.
\end{proof}

\subsection{Relationship with the mutation of ($\tau$-)tilting modules}

The final result of this article establishes a close relationship between the mutation of tilting $\A_t$-modules and the mutation of ($\tau$-)exceptional sequences in $\mods \A_t$. First, it is necessary to establish the following result, again setting up the inductive step of the proof. 

\begin{lemma}\label{lem:tiltreduction}
	Let $T$ be a basic tilting $\A_t$-module. Then $T_t/T_1 \oplus \dots \oplus T_2/T_1$ is a basic tilting object in $J(T_1)$. 
\end{lemma}
\begin{proof}
	First of all $\dim \Hom(T_1, P(1)) = 1 = \dim \Hom(P(1), P(i))$ implies $\dim \Hom(T_1, P(i)) = 1$ and since any nonzero morphism in $\Hom(P(1),P(i))$ and $\Hom(T_1, P(1))$ is a monomorphism, it follows that any nonzero morphism from $T_1$ to $P(i)$ is a monomorphism for all $i \in \{1, \dots, t\}$. 
	
	It then follows from a proof entirely analogous to that of \cref{lem:torfdesc} that for every summand $T_i \in \add T$, the inclusion map $\alpha_i: T_1 \to T_i $ is a minimal right $\Gen(T_1)$-approximation. It follows that $f_{T_1} (T_i) \simeq T_i/T_1$ for all $i=\{2,\dots, t\}$. 
	
	By \cite[Lem. 7.1]{DlabRingel1992}, the class of modules filtered by the standard modules of (the quasi-hereditary algebra) $\A_t$ is closed under submodules and contains all projective $\A_t$-modules. By \cite[Lem. 1]{BHRR1999}, this class of modules is equivalent to the class of $\A_t$-modules of projective dimension at most 1. In particular, every submodule of a projective $\A_t$-module has projective dimension at most 1. 
	
	From the chain of inclusions $T_1 \subset T_i \subset T_t = P(t)$ it follows that $f_{T_1}(T_i) = T_i/T_1$ is a submodule of $f_{T_1}(T_t) = T_t/T_1$ for all $i \in \{2, \dots, t\}$. By \cite[Prop. 4.5]{BuanMarsh2018t}, the modules $T_i/T_1$ and $T_t/T_1$ are $\tau$-rigid in $J(T_1)$. Moreover, sine $T_t= P(t)$ is projective, it is projective in ${}^\perp \tau T_1$, whence $T_t/T_1$ is projective in $J(T_1)$ by \cite[Prop. 3.14]{Jasso2015}, see also \cite[Lem. 4.9]{BuanMarsh2018t}. By the above, $T_i/T_1$ is therefore a subobject of a projective object in $J(T_1)$. Since $J(T_1) \simeq \mods \A_{t-1}$ by \cref{lem:reduction}, it follows that $T_i/T_1$ has projective dimension at most 1 in $J(T_1)$ for all $i \in \{2, \dots, t\}$. Therefore, it is partial tilting in $J(T_1)$ and the result follows.
\end{proof}

The previous lemma enables the proof of the final result of this article, establishing compatibility between the different kinds of mutation.

\begin{proposition}\label{prop:mutcomp}
	Let $T$ be a basic tilting $\A_t$-module and let $i \in \{2, \dots, t\}$. Assume that the complete exceptional sequence $\Phi(T)$ is left $\psi$-mutable at position $i$. Then there exists a tilting $\A_t$-module $T' \in \Gen T$ which differs from $T$ only in the direct summand $T_{i-1}$ and which is such that 
	\[ \Phi(T') = \psi_{i}(\Phi(T)) = \varphi_i(\Phi(T)). \]
	Conversely, assume that there exists a tilting $\A_t$-module $T'$ which differs from $T$ only in the summand $T_{i-1}$ and which is such that $T' \in \Gen T$. Then $\Phi(T)$ is left $\psi$-mutable at position $i$. 
\end{proposition}
\begin{proof}	
	Assume that $\Phi(T)$ is left $\psi$-mutable at position $i$. Consider first the case $i=2$. Assume for a contradiction that $T_1 \in \Gen(T/T_1)$. Then \cref{lem:tiltapprox}(2) implies that there is an epimorphism $f: T_2 \twoheadrightarrow T_1$. 
	Recall the equivalence of categories $\add(T) \simeq \add(\A_t)$ induced by \cref{thm:tiltingendo}, from which it follows that
	\begin{equation}\label{eq:mutcompdim} \dim \Hom(T_2, T_1) = \dim \Hom(P(2), P(1)) = 1,\end{equation}
	
	Denote by $\ell$ the composition of morphisms
	\[ \begin{tikzcd} T_2 \arrow[r, twoheadrightarrow, "\ell_1"] & f_{T_1} (T_2)  \arrow[r, hookrightarrow, "\ell_2"] & T_1  \end{tikzcd}, \]
	where a nonzero $\ell_2$ exists because $\Phi(T)$ is left $\psi$-mutable at position $2$ and a nonzero $\ell_1$ exists since $T_1$ is a proper submodule of $T_2$ by \cref{lem:tiltapprox}(1). By \cref{eq:mutcompdim}, there exists some nonzero $\lambda \in K$ such that $f = \lambda \ell$. However, since $f$ is an epimorphism, the morphism $\ell_2$ must be an epimorphism as well, a contradiction. It follows that $T_1 \not \in \Gen(T/T_1)$, and thus ${}^\perp \tau(T/T_1) \subseteq {}^\perp \tau T_1$ by \cite[Prop. 2.22]{AIR2014}.
	
	Consider a short exact sequence
	\[ 0 \to T_1 \xrightarrow{g} T_2 \to T_2/T_1 \to 0 \]
	where $g$ is a minimal left $\add(T/T_1)$-approximation of $T_1$ by \cite[Lem. 4.3(2)]{IyamaZhang2020}. Since ${}^\perp \tau(T/T_1) \subseteq {}^\perp \tau T_1$, it is possible to apply \cite[Thm. 2.30]{AIR2014} to obtain that the module 
	\[ T' \coloneqq T_t \oplus \dots \oplus T_2 \oplus T_2/T_1 \]
	is a basic $\tau$-tilting $\A_t$-module, which is known as the left (AIR-)mutation of $T$ at $T_1$. Clearly $T' \in \Gen T$. Since $\projdim (T_2/T_1) \leq 1$ by \cref{lem:completeexcepprop}(4), it follows that $\projdim T' \leq 1$ and hence $T'$ is a tilting $\A_t$-module by \cref{lem:tiltistau}. Moreover, \cref{lem:torfdesc} yields 
	\[ \Phi(T') = (T_t/T_{t-1}, \dots T_3/T_2 , f_{T_2/T_1} (T_2) , T_2/T_1). \]
	By \cref{thm:excepistauexcep}, this is a ($\tau$-)exceptional sequence which differs from the ($\tau$-)exceptional sequence 
	\[ \psi_2(\Phi(T)) = (T_t/T_{t-1}, \dots, T_3/T_2, \mathcal{L}_{T_2/T_1} T_1, T_2/T_1) \]
	in at most one position. By \cite[Thm. 2.1(a)]{BuanMarsh2023}, see also \cite[Thm. 8]{HansonThomas2024}, they therefore coincide. In conclusion
	\[ \Phi(T') = \psi_2 (\Phi(T)) = \varphi_2 (\Phi(T)). \]
	
	The general case for $i \in \{3, \dots, t\}$ follows entirely analogously in the subcategory $J(T_{i-2} \oplus \dots \oplus T_1)$, which is isomorphic to $\mods \A_{t-i+2}$ by \cref{lem:reduction}. Applying $f_{T_{i-2}}$ to each summand of $T/(T_{i-2} \oplus \dots \oplus T_1)$ yields a tilting object in the subcategory by iteratively invoking \cref{lem:tiltreduction}, see \cite[Cor. 1.7]{BuanMarsh2021w} about the associativity of this process. This yields the setup of the $i=2$ case. In this case, the assumption that $T_{i-1} \in \Gen(T/T_{i-1})$ implies that there is an epimorphism $\epsilon: T_{i-2} \oplus T_i \to T_{i-1}$ by \cref{lem:tiltapprox}. Note that $f_{T_{i-2}}$ is a functor, which yields an epimorphism $f_{T_{i-2}}(\epsilon): f_{T_{i-2}} T_i \to f_{T_{i-2}} T_{i-1}$.  from where the proof carries over exactly as before in the subcategory $J(T_{i-2} \oplus \dots \oplus T_1)$.

	For the converse, assume that a tilting $\A_t$-module $T' \in \Gen T$ exists which differs from $T$ only in the summand $T_{i-1}$ . Consider first the case $i=2$. Assume that $\Phi(T)$ cannot be $\psi$-mutated at position $2$. By \cref{lem:completeexcepprop}(1) and (3) it follows that there is a nonzero surjection $T_2/T_1 \to T_1$. This means that $T_1 \in \Gen T_2 \subseteq \Gen (T/T_1)$. However, since $T' \in \Gen T$ also implies $\Gen T' \subseteq T$, this leads to the chain of inclusions
	\[ \Gen T' \subseteq \Gen T \subseteq \Gen(T/T_1) \subseteq \Gen(T').  \]
	Therefore $\Gen T' = \Gen T$, which implies $T= T'$ by \cite[Thm. 2.7]{AIR2014} as both $T$ and $T'$ are ($\tau$-)tilting modules. This is a contradiction. For the general case, iteratively applying \cref{lem:tiltreduction}, allows the use of the same argument in the subcategory $J(T_{i-2} \oplus \dots \oplus T_1)$, which is isomorphic to $\mods \A_{t-i+2}$ by \cref{lem:reduction}. This completes the proof.
\end{proof}

\section{Example}\label{sec:example}
This section illustrates the main theorems of this article for the case $t=3$. Recall from \cref{sec:prelim} that the algebra $\A_3$ is isomorphic to 
\[ K \left( \begin{tikzcd} 1 \arrow[r, "a_1", shift left] & 2 \arrow[r, "a_2", shift left] \arrow[l, "b_1", shift left] & 3 \arrow[l, "b_2", shift left] \end{tikzcd} \right) / \langle a_1 b_1, a_2b_2 - b_1a_1 \rangle . \]
In \cref{fig:stautiltgraph} and \cref{fig:A3seqs}, indecomposable $\A_3$-modules are written via their composition series. First of all, consider the tilting $\A_3$-modules highlighted in {\color{purple} \textbf{purple}} in \cref{fig:stautiltgraph}. By \cref{lem:tiltistau}, these are also $\tau$-tilting modules and they embed into the exchange quiver of $\tau$-tilting pairs, as illustrated in \cref{fig:stautiltgraph}. In \cref{fig:stautiltgraph}, an arrow is drawn from a $\tau$-tilting pair $T \oplus P$ to another $\tau$-tilting module pair $T' \oplus P'$ whenever $\Gen T'  \subset \Gen T$ holds.

\begin{figure}[ht!]
\[
\begin{tikzcd}[ampersand replacement =\&, column sep=15, row sep = 30]
    \& \& {\color{purple}\boldsymbol{\begin{smallmatrix} &&3\\&2&\\1&&3\\&2&\\&&3 \end{smallmatrix}  \oplus \begin{smallmatrix} &2&\\1&&3\\&2&\\&&3 \end{smallmatrix} \oplus \begin{smallmatrix} 1&&\\&2&\\&&3 \end{smallmatrix}}} \arrow[ld,purple, thick, "\mu_1"] \arrow[d, purple, thick, "\mu_2",swap] \arrow[lld] \\
    \begin{smallmatrix} &2\\1& \end{smallmatrix}  \oplus \begin{smallmatrix} &2&\\1&&3\\&2&\\&&3 \end{smallmatrix}  \oplus \begin{smallmatrix} 1&&\\&2&\\&&3 \end{smallmatrix} \arrow[d] \& {\color{purple} \boldsymbol{\begin{smallmatrix} &&3\\&2&\\1&&3\\&2&\\&&3 \end{smallmatrix} \oplus \begin{smallmatrix} &2&\\1&&3\\&2&\\&&3 \end{smallmatrix} \oplus \begin{smallmatrix} 2&\\&3 \end{smallmatrix}}} \arrow[d, purple, thick, "\mu_2",] \arrow[ld] \& {\color{purple} \boldsymbol{ \begin{smallmatrix} &&3\\&2&\\1&&3\\&2&\\&&3 \end{smallmatrix} \oplus \begin{smallmatrix} 1&&3\\&2&\\&&3 \end{smallmatrix} \oplus \begin{smallmatrix} 1&&\\&2&\\&&3 \end{smallmatrix}}} \arrow[d, purple, thick, "\mu_1",swap] \arrow[rd] \&\\
    \begin{smallmatrix} &2\\1& \end{smallmatrix} \oplus \begin{smallmatrix} &2&\\1&&3\\&2&\\&&3 \end{smallmatrix} \oplus \begin{smallmatrix} 2&\\&3 \end{smallmatrix}  \& {\color{purple} \boldsymbol{\begin{smallmatrix} &&3\\&2&\\1&&3\\&2&\\&&3 \end{smallmatrix} \oplus \begin{smallmatrix} &&3\\&2&\\&&3 \end{smallmatrix} \oplus \begin{smallmatrix} 2&\\&3 \end{smallmatrix}}} \arrow[d, purple, thick, "\mu_1"] \arrow[ld] \& {\color{purple} \boldsymbol{ \begin{smallmatrix} &&3\\&2&\\1&&3\\&2&\\&&3 \end{smallmatrix} \oplus \begin{smallmatrix} 1&&3\\&2&\\&&3 \end{smallmatrix} \oplus \begin{smallmatrix} 3 \end{smallmatrix}}} \arrow[ld, purple, thick, "\mu_2", swap] \arrow[d] \& \begin{smallmatrix} 1 \end{smallmatrix} \oplus \begin{smallmatrix} 1&&3\\&2&\\&&3 \end{smallmatrix} \oplus \begin{smallmatrix} 1&&\\&2&\\&&3 \end{smallmatrix} \arrow[ld] \arrow[d] \\
    \begin{smallmatrix} &&3\\&2&\\&&3 \end{smallmatrix} \oplus \begin{smallmatrix} 2&\\&3 \end{smallmatrix} \oplus P(1)[1] \arrow[rd] \arrow[d] \& {\color{purple}\boldsymbol{ \begin{smallmatrix} &&3\\&2&\\1&&3\\&2&\\&&3 \end{smallmatrix} \oplus \begin{smallmatrix} &&3\\&2&\\&&3 \end{smallmatrix} \oplus \begin{smallmatrix} 3 \end{smallmatrix}}} \arrow[d] \& \begin{smallmatrix} 1 \end{smallmatrix} \oplus \begin{smallmatrix} 1&&3\\&2&\\&&3 \end{smallmatrix} \oplus \begin{smallmatrix} 3 \end{smallmatrix} \arrow[d] \& \begin{smallmatrix} 1 \end{smallmatrix} \oplus \begin{smallmatrix} 1&\\&2 \end{smallmatrix} \oplus \begin{smallmatrix} 1&&\\&2&\\&&3 \end{smallmatrix} \arrow[d] \\
     \begin{smallmatrix} 3 \end{smallmatrix} \oplus \begin{smallmatrix} 2&\\&3 \end{smallmatrix}\oplus P(1)[1]  \& \begin{smallmatrix} &&3\\&2&\\&&3 \end{smallmatrix} \oplus \begin{smallmatrix} 3 \end{smallmatrix} \oplus P(1)[1] \& \begin{smallmatrix} 1 \end{smallmatrix} \oplus \begin{smallmatrix} 3 \end{smallmatrix} \oplus P(2)[1] \& \begin{smallmatrix} 1 \end{smallmatrix} \oplus \begin{smallmatrix} 1&\\&2 \end{smallmatrix} \oplus P(3)[1]
\end{tikzcd}
\]
    \caption{The embedding of tilting $\A_3$-modules (\color{purple}purple\color{black}) into part of the exchange quiver of $\tau$-tilting pairs in $\mods \A_3$. Tilting modules are written in their quasi-hereditary decomposition.}
    \label{fig:stautiltgraph}
\end{figure}
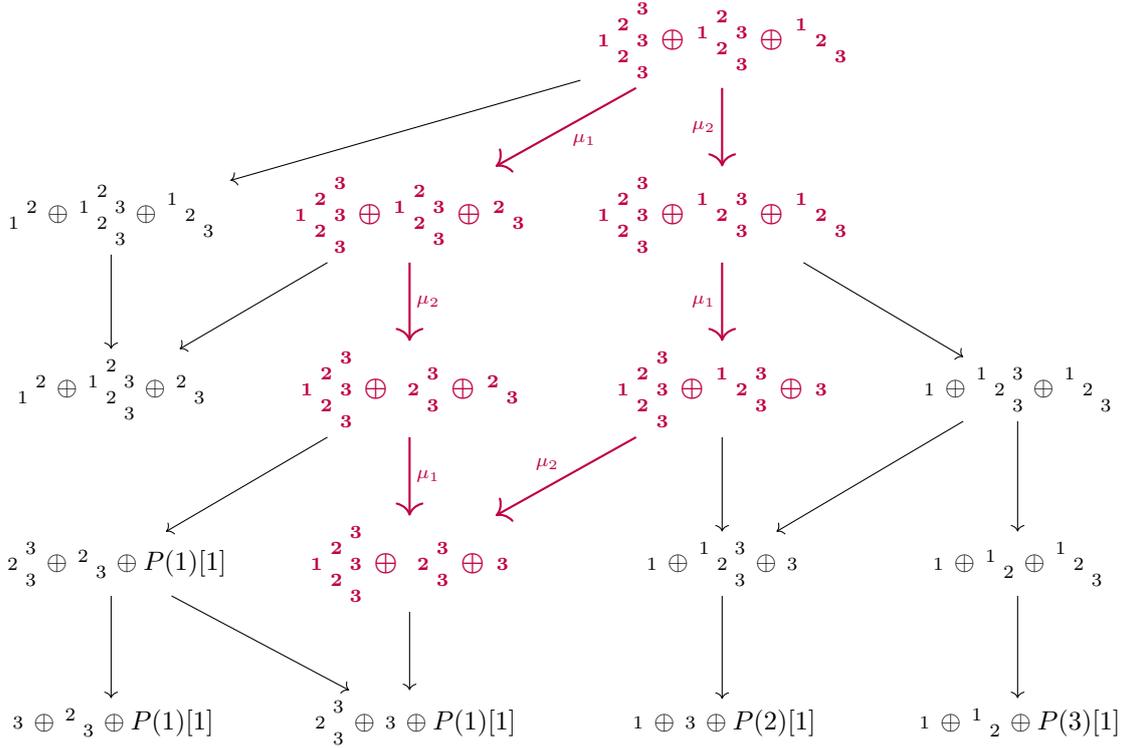

Moreover, \cref{thm:excepistauexcep} shows that from each basic tilting $\A_3$-module it is possible to construct a complete exceptional sequence in $\mods \A_3$ via the bijection $\Phi$ of \cref{thm:MendozaTreff}. In addition, every complete exceptional sequence in $\mods \A_3$ arises in this way and by the definition of $\Phi$, all of these complete exceptional sequences are $\tau$-exceptional sequences. It was shown in \cref{thm:mutthm} that the ($\psi$-)mutation of complete exceptional sequences in $\mods \A_t$ is a special case of the mutation of ($\varphi$-)exceptional sequences. This is illustrated in \cref{fig:A3seqs}, where exceptional sequences are highlighted {\color{purple} \textbf{purple}} and arrows are drawn to indicate left mutations between ($\tau$-)exceptional sequences.

Comparing the two figures, one sees that the subsets corresponding to tilting modules, respectively exceptional sequences, have the same shape. This illlustrates the behaviour established by \cref{prop:mutcomp}.

\begin{figure}[ht!]
\[ 
\begin{tikzcd}[ampersand replacement =\&, row sep=25]
    \left(\begin{smallmatrix} 3 \end{smallmatrix}, \begin{smallmatrix} 1 \end{smallmatrix}, \begin{smallmatrix} &2&\\1&&3\\&2&\\&&3 \end{smallmatrix}\right) \arrow[r, "\varphi_2",swap]  \& \left(\begin{smallmatrix} 3 \end{smallmatrix}, \begin{smallmatrix} 1&&\\&2&\\ &&3\end{smallmatrix}, \begin{smallmatrix} 1 \end{smallmatrix}\right) \arrow[r, "\varphi_2",swap] \& {\color{purple}\boldsymbol{ \left(\begin{smallmatrix} 3 \end{smallmatrix} , \begin{smallmatrix} 2&\\&3 \end{smallmatrix} , \begin{smallmatrix} 1&&\\&2&\\&&3 \end{smallmatrix}\right)}} \arrow[ld, thick, purple, "\varphi_2"', "\psi_2"] \arrow[d, thick, purple, "\varphi_3", "\psi_3"'] \& \left(\begin{smallmatrix} 2&\\&3 \end{smallmatrix}, \begin{smallmatrix} 2 \end{smallmatrix}, \begin{smallmatrix} 1&&\\&2&\\&&3 \end{smallmatrix}\right) \arrow[l, "\varphi_3"] \\
    \left(\begin{smallmatrix} &2&\\1&&3 \end{smallmatrix}, \begin{smallmatrix} &2\\1& \end{smallmatrix}, \begin{smallmatrix} 2&\\&3 \end{smallmatrix}\right) \arrow[r, "\varphi_3",swap] \& {\color{purple}\boldsymbol{\left(\begin{smallmatrix} 3 \end{smallmatrix} , \begin{smallmatrix} &2&\\1&&3 \end{smallmatrix}, \begin{smallmatrix} 2&\\&3 \end{smallmatrix}\right)}} \arrow[d, thick, purple, "\varphi_3"',"\psi_3"] \arrow[lu, "\varphi_2",swap] \& {\color{purple}\boldsymbol{\left(\begin{smallmatrix} &3\\2& \end{smallmatrix} , \begin{smallmatrix} 3 \end{smallmatrix}, \begin{smallmatrix} 1&&\\&2&\\&&3 \end{smallmatrix}\right)}}  \arrow[d, purple, thick, "\varphi_2", "\psi_2"'] \arrow[r, "\varphi_3"] \& \left(\begin{smallmatrix} 2 \end{smallmatrix}, \begin{smallmatrix} &3\\2&\\&3 \end{smallmatrix}, \begin{smallmatrix} 1&&\\&2&\\&&3 \end{smallmatrix}\right) \arrow[u, "\varphi_3"] \\
   \left(\begin{smallmatrix} &2\\1& \end{smallmatrix}, \begin{smallmatrix} &&3\\&2&\\1&&3 \end{smallmatrix}, \begin{smallmatrix} 2&\\&3 \end{smallmatrix}\right) \arrow[u, "\varphi_3",swap] \& {\color{purple}\boldsymbol{\left(\begin{smallmatrix} &&3\\&2&\\1&& \end{smallmatrix} , \begin{smallmatrix} 3 \end{smallmatrix}, \begin{smallmatrix} 2&\\&3 \end{smallmatrix}\right)}} \arrow[d, thick, purple, "\varphi_2"', "\psi_2"] \arrow[l, "\varphi_3",swap] \& {\color{purple}\boldsymbol{\left(\begin{smallmatrix} &3\\2& \end{smallmatrix} , \begin{smallmatrix} 1&&3\\&2& \end{smallmatrix}, \begin{smallmatrix} 3 \end{smallmatrix}\right)}} \arrow[ld, thick, purple, "\varphi_3", "\psi_3"'] \arrow[r, "\varphi_2"] \& \left(\begin{smallmatrix} &3\\2& \end{smallmatrix}, \begin{smallmatrix} 1&\\&2 \end{smallmatrix}, \begin{smallmatrix} 1&&3\\&2&\\&&3 \end{smallmatrix}\right) \arrow[lu, "\varphi_2"] \\
    \left(\begin{smallmatrix} &&3\\&2&\\1&& \end{smallmatrix}, \begin{smallmatrix} 2 \end{smallmatrix}, \begin{smallmatrix} &&3\\&2&\\&&3 \end{smallmatrix}\right) \arrow[ru, "\varphi_2",swap] \& {\color{purple}\boldsymbol{\left(\begin{smallmatrix} &&3\\&2&\\1&& \end{smallmatrix} , \begin{smallmatrix} &3\\2& \end{smallmatrix}, \begin{smallmatrix} 3 \end{smallmatrix}\right)}} \arrow[l, "\varphi_2",swap] \arrow[r, "\varphi_3"] \& \left(\begin{smallmatrix} 1 \end{smallmatrix}, \begin{smallmatrix} &&3\\&2&\\1&&3\\&2& \end{smallmatrix}, \begin{smallmatrix} 3 \end{smallmatrix}\right) \arrow[r, "\varphi_3"] \& \left(\begin{smallmatrix} 1&&3\\&2& \end{smallmatrix}, \begin{smallmatrix} 1 \end{smallmatrix}, \begin{smallmatrix} 3 \end{smallmatrix}\right) \arrow[lu, "\varphi_3"]
\end{tikzcd}\]
    \caption{A subset of the complete $\tau$-exceptional sequences of $\mods \A_3$ with arrows corresponding to left $\varphi$-mutations. The set of complete exceptional sequences and the corresponding $\psi$-mutation in $\mods \A_3$ form a subset which is highlighted in {\color{purple}\textbf{purple}}.}
    \label{fig:A3seqs}
\end{figure}
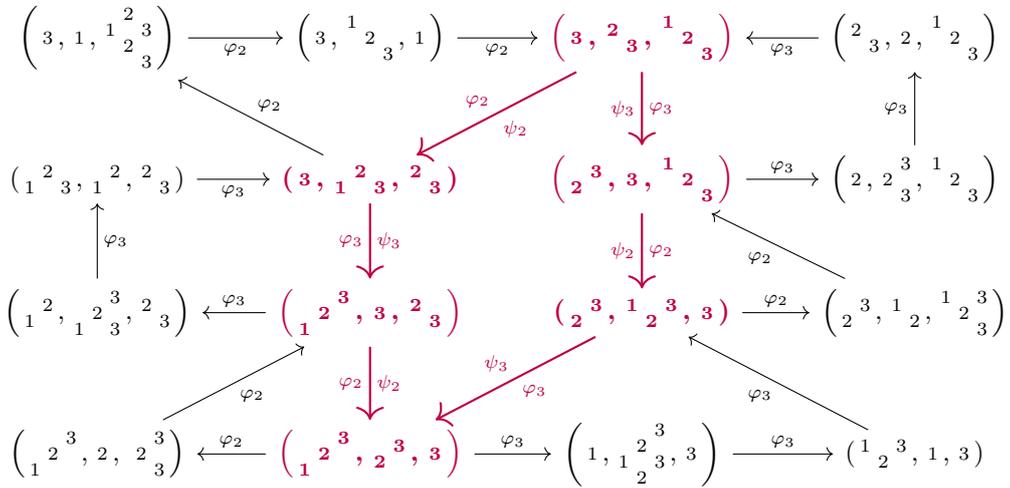

\bibliographystyle{ourIEEEstyle.bst}
\bibliography{main}

@article{HillePloog2019,
 author = {Hille, Lutz and Ploog, David},
 title = {Exceptional sequences and spherical modules for the {Auslander} algebra of $k[x]/(x^t)$},
 fjournal = {Pacific Journal of Mathematics},
 journal = {Pac. J. Math.},
 volume = {302},
 number = {2},
 pages = {599--625},
 year = {2019},
}

@misc{Gorodentsev1990,
 author = {Gorodentsev, A. L.},
 title = {Exceptional objects and mutations in derived categories},
 year = {1990},
 language = {English},
 howpublished = {Helices and vector bundles, {Proc}. {Semin}. {Rudakov}, {Lond}. {Math}. {Soc}. {Lect}. {Note} {Ser}. 148, 57-73},
}

@article{MendozaTreffinger,
 author = {Mendoza, Octavio and Treffinger, Hipolito},
 title = {Stratifying systems through $\tau$-tilting theory},
 fjournal = {Documenta Mathematica},
 journal = {Doc. Math.},
 volume = {25},
 pages = {701--720},
 year = {2020},
}

@misc{BHM2024,
      title={Mutation of $\tau$-exceptional pairs and sequences}, 
      author={Aslak B. Buan and Eric J. Hanson and Bethany R. Marsh},
      year={2024},
      eprint={2402.10301},
      archivePrefix={arXiv},
      primaryClass={math.RT},
      url={https://arxiv.org/abs/2402.10301}, 
}

@article{BuanMarsh2018t,
 author = {Buan, Aslak Bakke and Marsh, Bethany Rose},
 title = {$\tau$-exceptional sequences},
 fjournal = {Journal of Algebra},
 journal = {J. Algebra},
 volume = {585},
 pages = {36--68},
 year = {2021},
}

@article{IyamaZhang2020,
 author = {Iyama, Osamu and Zhang, Xiaojin},
 title = {Classifying $\tau$-tilting modules over the {Auslander} algebra of $K[x]/(x^n)$},
 fjournal = {Journal of the Mathematical Society of Japan},
 journal = {J. Math. Soc. Japan},
 volume = {72},
 number = {3},
 pages = {731--764},
 year = {2020},
}

@book{ARS1995,
 author = {Auslander, Maurice and Reiten, Idun and Smal{\o}, Sverre O.},
 title = {Representation theory of {Artin} algebras},
 fseries = {Cambridge Studies in Advanced Mathematics},
 series = {Camb. Stud. Adv. Math.},
 volume = {36},
 year = {1995},
 publisher = {Cambridge: Cambridge University Press},
}

@article{BHRR1999,
 author = {Br{\"u}stle, Thomas and Hille, Lutz and Ringel, Claus M. and R{\"o}hrle, Gerhard},
 title = {The {$\Delta$}-filtered modules without self-extensions for the {Auslander} algebra of $k[T]/\langle T^n \rangle$},
 fjournal = {Algebras and Representation Theory},
 journal = {Algebr. Represent. Theory},
 volume = {2},
 number = {3},
 pages = {295--312},
 year = {1999},
}

@misc{BKT2025,
      title={Mutating ordered $\tau$-rigid modules with applications to {N}akayama algebras}, 
      author={Aslak B. Buan and Maximilian Kaipel and Håvard U. Terland},
      year={2025},
      eprint={2501.13694},
      archivePrefix={arXiv},
      primaryClass={math.RT},
      url={https://arxiv.org/abs/2501.13694}, 
}

@article{BuanMarsh2023,
 author = {Buan, Aslak B. and Marsh, Bethany R.},
 title = {Mutating signed $\tau$-exceptional sequences},
 fjournal = {Glasgow Mathematical Journal},
 journal = {Glasg. Math. J.},
 volume = {65},
 number = {3},
 pages = {716--729},
 year = {2023},
}

@article{HansonThomas2024,
 author = {Hanson, Eric J. and Thomas, Hugh},
 title = {A uniqueness property of $\tau$-exceptional sequences},
 fjournal = {Algebras and Representation Theory},
 journal = {Algebr. Represent. Theory},
 volume = {27},
 number = {1},
 pages = {461--468},
 year = {2024},
}

@article{HillePerling2014,
 author = {Hille, Lutz and Perling, Markus},
 title = {Tilting bundles on rational surfaces and quasi-hereditary algebras},
 fjournal = {Annales de l'Institut Fourier},
 journal = {Ann. Inst. Fourier},
 volume = {64},
 number = {2},
 pages = {625--644},
 year = {2014},
}

@article{BST2019,
 author = {Br{\"u}stle, Thomas and Smith, David and Treffinger, Hipolito},
 title = {Wall and chamber structure for finite-dimensional algebras},
 fjournal = {Advances in Mathematics},
 journal = {Adv. Math.},
 volume = {354},
 pages = {31},
 note = {Id/No 106746},
 year = {2019},
}

@misc{LiZhang2015,
      title={Some applications of $\tau$-tilting theory}, 
      author={Shen Li and Shunhua Zhang},
      year={2015},
      eprint={1512.03613},
      archivePrefix={arXiv},
      primaryClass={math.RT},
      url={https://arxiv.org/abs/1512.03613}, 
}

@article{AIR2014,
 author = {Adachi, Takahide and Iyama, Osamu and Reiten, Idun},
 title = {$\tau$-tilting theory.},
 fjournal = {Compositio Mathematica},
 journal = {Compos. Math.},
 volume = {150},
 number = {3},
 pages = {415--452},
 year = {2014},
}

@article{Auslander1974b,
 author = {Auslander, Maurice},
 title = {Representation theory of {Artin} algebras. {II}},
 fjournal = {Communications in Algebra},
 journal = {Commun. Algebra},
 issn = {0092-7872},
 volume = {1},
 pages = {269--310},
 year = {1974},
}

@article{Geuenich2022,
 author = {Geuenich, Jan},
 title = {Tilting modules for the {Auslander} algebra of $K[x]/(x^n)$},
 fjournal = {Communications in Algebra},
 journal = {Commun. Algebra},
 issn = {0092-7872},
 volume = {50},
 number = {1},
 pages = {82--95},
 year = {2022},
}

@article{Dickson66,
 author = {Dickson, Spencer E.},
 title = {A torsion theory for {Abelian} categories},
 fjournal = {Transactions of the American Mathematical Society},
 journal = {Trans. Am. Math. Soc.},
 issn = {0002-9947},
 volume = {121},
 pages = {223--235},
 year = {1966},
}

@article{BuanMarsh2021w,
 author = {Buan, Aslak Bakke and Marsh, Bethany Rose},
 title = {A category of wide subcategories},
 fjournal = {IMRN. International Mathematics Research Notices},
 journal = {Int. Math. Res. Not.},
 issn = {1073-7928},
 volume = {2021},
 number = {13},
 pages = {10278--10338},
 year = {2021},
}

@article{BuanHanson2023,
 author = {Buan, Aslak Bakke and Hanson, Eric J.},
 title = {$\tau$-perpendicular wide subcategories},
 fjournal = {Nagoya Mathematical Journal},
 journal = {Nagoya Math. J.},
 issn = {0027-7630},
 volume = {252},
 pages = {959--984},
 year = {2023},
}

@article{Jasso2015,
    author = {Jasso, Gustavo},
    title = {Reduction of $\tau$-tilting modules and torsion pairs},
    journal = {Int. Math. Res. Not.},
    volume = {2015},
    number = {16},
    pages = {7190-7237},
    year = {2015}
}

@article{Kapranov1988,
 author = {Kapranov, M. M.},
 title = {On the derived categories of coherent sheaves on some homogeneous spaces},
 fjournal = {Inventiones Mathematicae},
 journal = {Invent. Math.},
 issn = {0020-9910},
 volume = {92},
 number = {3},
 pages = {479--508},
 year = {1988},
}

@article{Beilinson1979,
 author = {Beilinson, A. A.},
 title = {Coherent sheaves on $\mathbb{P}^n$ and problems of linear algebra},
 fjournal = {Functional Analysis and its Applications},
 journal = {Funct. Anal. Appl.},
 issn = {0016-2663},
 volume = {12},
 pages = {214--216},
 year = {1979},
}

@InCollection{cbw92,
 Author = {Crawley-Boevey, William},
 Title = {Exceptional sequences of representations of quivers},
 BookTitle = {Representations of algebras. (Ottawa, ON, 1992)}, 
 Pages = {117--124},
series ={CMS Conf. Proc., vol. 14},
 Year = {1993},
 Publisher = {Amer. Math. Soc., Providence, RI},
}

@article{Macri2007,
 author = {Macrì, Emanuele},
 title = {Stability conditions on curves},
 fjournal = {Mathematical Research Letters},
 journal = {Math. Res. Lett.},
 issn = {1073-2780},
 volume = {14},
 number = {4},
 pages = {657--672},
 year = {2007},
}

@article{PrabhuNaik2017,
 author = {Prabhu-Naik, Nathan},
 title = {Tilting bundles on toric {Fano} fourfolds},
 fjournal = {Journal of Algebra},
 journal = {J. Algebra},
 issn = {0021-8693},
 volume = {471},
 pages = {348--398},
 year = {2017},
}

@article{CHS2025,
 author = {Chang, Wen and Haiden, Fabian and Schroll, Sibylle},
 title = {Braid group actions on branched coverings and full exceptional sequences},
 fjournal = {Advances in Mathematics},
 journal = {Adv. Math.},
 issn = {0001-8708},
 volume = {472},
 pages = {24},
 note = {Id/No 110284},
 year = {2025},
}

@article{Kawamata2006,
 author = {Kawamata, Yujiro},
 title = {Derived categories of toric varieties},
 fjournal = {Michigan Mathematical Journal},
 journal = {Mich. Math. J.},
 issn = {0026-2285},
 volume = {54},
 number = {3},
 pages = {517--535},
 year = {2006},
}

@article{Bondal1990,
 author = {Bondal, A. I.},
 title = {Representation of associative algebras and coherent sheaves},
 fjournal = {Mathematics of the USSR. Izvestiya},
 journal = {Math. USSR, Izv.},
 issn = {0025-5726},
 volume = {34},
 number = {1},
 pages = {23--42},
 year = {1990},
}

@article{Krause2017,
 author = {Krause, Henning},
 title = {Highest weight categories and recollements},
 fjournal = {Annales de l'Institut Fourier},
 journal = {Ann. Inst. Fourier},
 issn = {0373-0956},
 volume = {67},
 number = {6},
 pages = {2679--2701},
 year = {2017},
}

@article{HilleRoehrle1999,
 author = {Hille, L. and Röhrle, G.},
 title = {A classification of parabolic subgroups of classical groups with a finite number of orbits on the unipotent radical},
 fjournal = {Transformation Groups},
 journal = {Transform. Groups},
 issn = {1083-4362},
 volume = {4},
 number = {1},
 pages = {35--52},
 year = {1999},
}

@article{HillePloog2019b,
 author = {Hille, Lutz and Ploog, David},
 title = {Tilting chains of negative curves on rational surfaces},
 fjournal = {Nagoya Mathematical Journal},
 journal = {Nagoya Math. J.},
 issn = {0027-7630},
 volume = {235},
 pages = {26--41},
 year = {2019},
}

@misc{Nonis2025,
      title={$\tau$-exceptional sequences for representations of quivers over local algebras}, 
      author={Iacopo Nonis},
      year={2025},
      eprint={2502.15417},
      archivePrefix={arXiv},
      primaryClass={math.RT},
      url={https://arxiv.org/abs/2502.15417}, 
}

@misc{Nonis2025b,
      title={Mutation of $\tau$-exceptional sequences for acyclic quivers over local algebras}, 
      author={Iacopo Nonis},
      year={2025},
      eprint={2505.22770},
      archivePrefix={arXiv},
      primaryClass={math.RT},
      url={https://arxiv.org/abs/2505.22770}, 
}

@misc{BHM2025,
      title={Transitivity of mutation of $\tau$-exceptional sequences in the $\tau$-tilting finite case}, 
      author={Aslak B. Buan and Eric J. Hanson and Bethany R. Marsh},
      year={2025},
      eprint={2506.21372},
      archivePrefix={arXiv},
      primaryClass={math.RT},
      url={https://arxiv.org/abs/2506.21372}, 
}

@misc{KaipelTerland2025,
      title={Classifying {N}akayama algebras with a braid group action on $\tau$-exceptional sequences}, 
      author={Maximilian Kaipel and Håvard Utne Terland},
      year={2025},
      eprint={2507.07608},
      archivePrefix={arXiv},
      primaryClass={math.RT},
      url={https://arxiv.org/abs/2507.07608}, 
}

@Book{Rudakov90,
 Editor = {Rudakov, Alexei N.},
 Title = {Helices and vector bundles: seminaire {Rudakov}. {Transl}. by {A}. {D}. {King}, {P}. {Kobak} and {A}. {Maciocia}},
 FSeries = {London Mathematical Society Lecture Note Series},
 Series = {Lond. Math. Soc. Lect. Note Ser.},
 ISSN = {0076-0552},
 Volume = {148},
 Year = {1990},
 Publisher = {Cambridge University Press, Cambridge. London Mathematical Society, London},
 Language = {English},
}

@article {GorodentsevRudakov,
    AUTHOR = {Gorodentsev, Alexey L. and Rudakov, Alexei N.},
     TITLE = {Exceptional vector bundles on projective spaces},
   JOURNAL = {Duke Math. J.},
    VOLUME = {54},
      YEAR = {1987},
    NUMBER = {1},
     PAGES = {115--130},
}

@article {Gorodentsev1988,
    AUTHOR = {Gorodentsev, Alexey L.},
     TITLE = {Exceptional bundles on surfaces with a moving anticanonical class},
   JOURNAL = {Izv. Akad. Nauk SSSR Ser. Mat.},
    VOLUME = {52},
      YEAR = {1988},
    NUMBER = {4},
     PAGES = {740--757, 895},
}

@incollection{DlabRingel1992,
 author = {Dlab, Vlastimil and Ringel, Claus Michael},
 title = {The module theoretical approach to quasi-hereditary algebras},
 booktitle = {Representations of algebras and related topics. Proceedings of the Tsukuba international conference, held in Kyoto, Japan, 1990},
 pages = {200--224},
 year = {1992},
 publisher = {Cambridge: Cambridge University Press},
}

@article{FZ2002,
 author = {Fomin, Sergey and Zelevinsky, Andrei},
 title = {Cluster algebras. {I}: {Foundations}},
 fjournal = {Journal of the American Mathematical Society},
 journal = {J. Am. Math. Soc.},
 issn = {0894-0347},
 volume = {15},
 number = {2},
 pages = {497--529},
 year = {2002},
}

@article{RingelZhang2014,
 author = {Ringel, Claus Michael and Zhang, Pu},
 title = {From submodule categories to preprojective algebras.},
 fjournal = {Mathematische Zeitschrift},
 journal = {Math. Z.},
 issn = {0025-5874},
 volume = {278},
 number = {1-2},
 pages = {55--73},
 year = {2014},
}

@misc{Sauter2019,
      title={On the tilting complexes for the Auslander algebra of the truncated polynomial ring}, 
      author={Julia Sauter},
      year={2019},
      eprint={1907.04949},
      archivePrefix={arXiv},
      primaryClass={math.RA},
      url={https://arxiv.org/abs/1907.04949}, 
}

@book{Mitchell,
 author = {Mitchell, Barry},
 title = {Theory of categories},
 fseries = {Pure and Applied Mathematics (Academic Press)},
 series = {Pure Appl. Math., Academic Press},
 volume = {17},
 year = {1965},
 publisher = {New York {and} London: Academic Press},
 language = {English},
}

\end{document}